\documentclass[12pt,leqno,twoside]{amsart}

\usepackage{enumerate}

\usepackage{amssymb}

\usepackage{amssymb,amsthm,amsmath,eucal,mathrsfs,verbatim}

\usepackage{footnote}

\usepackage{tikz, xcolor}

\usepackage{cite}

\usepackage{amsmath}

\usepackage{amsthm}

\usepackage{amssymb}

\usepackage{amsfonts}
\usepackage{todonotes}

\usepackage{hyperref}
\usepackage{graphicx}   
\usepackage{tikz}       
\usepackage{caption}    
\usepackage{float}      
\usepackage{pgfplots}
\pgfplotsset{width=10cm,compat=1.9}
\usepgfplotslibrary{external}
\newtheorem{theorem}{Theorem}[section]

\newtheorem{lemma}[theorem]{Lemma}

\theoremstyle{definition}

\numberwithin{equation}{section}

\newcommand\restr[2]{{
  \left.\kern-\nulldelimiterspace 
  #1 
  \vphantom{\big|} 
  \right|_{#2} 
  }}

\usepackage{eucal}

\title[explicit formula for vanishing viscosity limit]{Viscous Burgers equation driven by point source: a formula for the weak limit}
\author[Pal and Sahoo\hfil\hfilneg]{Smritikana Pal and Manas R. Sahoo}
\email{}

\keywords{Burgers equation with point source; Vanishing viscosity limit; Lax-{O}le\u\i nik formula.\\
School of Mathematical Sciences, National Institute of Science Education and Research, An OCC of Homi Bhabha National Institute, Bhubaneswar, P.O. Jatni, Khurda, Odisha 752050, India. Email: smritikana.pal@niser.ac.in, manas@niser.ac.in}

\begin{document}

\begin{abstract}
In this article, we obtain the weak limit of the solutions of the viscous Burgers equation driven by a point source term, as the coefficient of viscosity tends to zero. The weak limit is related to the variational problem that consists of three types of functional, which is not usual in the absence of the source term. 
\end{abstract}
\maketitle 

\section{Introduction}
Conservation laws driven by a point source can be classified by the broader class of equations, namely
 \begin{equation}
 \label{general equation}
 u_t + F(t,x, u)_x= \alpha(t, x)\delta.
 \end{equation}
 Equations of the above type are useful for studying fluid flows in heterogeneous media, such as the fluid flow in changing rock type conditions relevant for the petroleum industry.
 This equation is closely related to the Hamilton-Jacobi equation with discontinuous flux. Hamilton-Jacobi equation with discontinuous flux function has been studied by many authors, see \cite{Adimurthi2007ExplicitHF,MR2028700,MR2299772, MR2323047}. The Starting point goes back to the work of E. Hopf\cite{hopf}, where he considered the above problem\eqref{general equation} for the case where $F(t, x, u)=\frac{u^2}{2}$ and $\alpha(t, x)=0$. He used the vanishing viscosity method and obtained the weak entropy solution to the problem \eqref{general equation} by passing to the limit as the viscosity coefficient approaches zero. Then Lax(\cite{pdlax}) and Ole\u{\i}nik
\cite{Oleinik1957, Oleinik1963} further advanced the subject. For example, Lax considered the problem \eqref{general equation} when $F(t, x, u)=f(u)$ is a convex function and $\alpha(t, x)=0$. He obtained an explicit formula for the entropy admissible weak solution. Ole\u{\i}nik studied the uniqueness by considering a one-sided Lipschitz condition. Recently, Chung, et al.\cite{chung} considered the following problem for $\epsilon=1.$

\begin{equation}\label{viscous equation}
\begin{aligned}
\begin{cases}
 & u_t+(\frac{u^2}{2})_x = \epsilon u_{xx}+ \delta,\,\,\,\,\,\, x\in \mathbb{R},\,\, t>0,\\
    &u(x, 0)=u_{0}(x),  x\in \mathbb{R}
 \end{cases}
    \end{aligned}
\end{equation}
where  $u_0$  is a bounded measurable function. Employing an interesting idea of the first order contact along the boundary, they obtained solution with less regularity along the boundary and studied the large time asymptotic of the above problem for $\epsilon=1$. Using the inverse Laplace transform, one can obtain the explicit formula for the solution of the viscous problem, see \cite{Satya}. Then the question is:``can one get the vanishing viscosity limit?" This is answered in this paper. The problem considered here can be written as the conservation law with discontinuous flux, which was well studied; see \cite{Adimurthi2007ExplicitHF,MR2028700,MR2299772, MR2323047}, and references therein.
Now we describe some of the interesting works related to this problem: 
Ostrov \cite{Ostrov} considered the problem  \begin{equation}\label{viscous equation 1}
        \begin{aligned}
            \begin{cases}
                v_t+F^{\delta}(x,v_x)=\epsilon v_{xx},&x\in \mathbb{R},t>0,\\
             v(x,0)=v_0(x),&x\in \mathbb{R}
            \end{cases}
        \end{aligned}
    \end{equation}  where $F^{\delta}$ is the regularization of $F$ and showed that the limit of  $v_{\epsilon, \delta}$, the unique solution of \eqref{viscous equation 1} converges to a unique viscosity solution $v$ of 
  \begin{equation}\label{inviscid equation 1}
        \begin{aligned}
            \begin{cases}
                v_t+F(x,v_x)=0,&x\in \mathbb{R},t>0,\\
             v(x,0)=v_0(x),&x\in \mathbb{R}
            \end{cases}
        \end{aligned}
    \end{equation} as $\epsilon \to 0, \delta \to 0.$
   Karlsen, Risebro, and Towers \cite{karlsen} showed that the solution of 
   \begin{equation*}
        \begin{aligned}
            \begin{cases}
                 u_t+F^{\delta}(x,u)_x=\epsilon u_{xx},&x\in \mathbb{R},t>0,\\
             u(x,0)=u_0(x),&x\in \mathbb{R}
            \end{cases}
        \end{aligned}
    \end{equation*} converges to a solution of \eqref{general equation} as $\epsilon \to 0, \delta \to 0$ and does not satisfy the interface entropy condition 
    \begin{equation}\label{condition}
        meas\{t:f^{\prime}(u^+(t))>0,g^{\prime}(u^-(t))<0\}=0
    \end{equation} where $F$ has a single point of discontinuity at $x=0$ and is given by $$F(x,u)=H(x)f(u)+(1-H(x))g(u)$$ and $f,g$ are Lipschitz-continuous functions and $H$ is the Heaviside function. 

From the aforementioned work, for example, karlsen et al. \cite{karlsen} studied the vanishing viscosity limit for the above problem. The vanishing viscosity approximations do not converge to the solution obtained by Adimurthi et al.\cite{Adimurthi2007ExplicitHF}. In this paper, we do not regularize the flux function and consider a less regular solution and show that our solution converges to the solution obtained by Adimurthi et al.\cite{Adimurthi2007ExplicitHF}. To obtain the solution for the viscous problem\eqref{viscous equation}, we use a method of 1st-order contact along the boundary, see\cite{chung}.

The paper is organized as follows: In section $2$, we obtain the explicit formula for vanishing viscosity approximation. In section $3,$ we obtain the distributional limit of the solutions $U^{\epsilon},$  where $U^{\epsilon}_x= u^{\epsilon}$ is the solution of the viscous problem \eqref{viscous equation}. In section $4$, we study some properties of the minimizer associated with the limit obtained in section-3 and also obtain an explicit formula for the limit function $u$. The vanishing viscosity limit agrees with the explicit formula obtained by Adimurthi et al.\cite{MR2028700} is indeed a weak solution of the inviscid equation, namely
\begin{equation}\label{inviscid equation}
\begin{aligned}
\begin{cases}
 & u_t+(\frac{u^2}{2})_x =\delta,\,\,\,\,\,\, x\in \mathbb{R},\,\, t>0,\\
    &u(x, 0)=u_{0}(x),  x\in \mathbb{R}.
 \end{cases}
    \end{aligned}
\end{equation}

\section{Explicit solution for viscous regularization}
For $\epsilon>0$, we obtain the explicit solution of the problem\eqref{viscous equation} described in the following theorem.
\begin{theorem}
The explicit solution $u^{\epsilon}, \epsilon>0$ of problem \eqref{viscous equation} is given by 
$u^{\epsilon}(x,t)=-\frac{2\epsilon \theta^{\epsilon}_x}{\theta^{\epsilon}},$
where $$\theta^{\epsilon}(x,t)=\begin{cases}
    R^{\epsilon}(x,t) & x>0, t>0\\
     L^{\epsilon}(x,t) & x<0, t>0,
\end{cases}$$
and  $R^{\epsilon},L^{\epsilon}$ are as follows:
\begin{equation*}
    \begin{aligned}
        R^{\epsilon}(x,t)&=e^{-\frac{t}{2\epsilon}}\Big[\frac{1}{2\sqrt{\pi \epsilon t}}\int_0^\infty \Big\{e^{-\frac{(x-\xi)^2}{4\epsilon t}}-e^{-\frac{(x+\xi)^2}{4\epsilon t}}\Big\}\theta^{\epsilon}_0(\xi)d\xi\\&+\int_0^t\Big(e^{\frac{\tau}{2\epsilon}}g(\tau)\Big)^\prime~~ erfc\Big(\frac{x}{2\sqrt{\epsilon(t-\tau)}}\Big)d\tau+g(0)erfc\Big(\frac{x}{2\sqrt{\epsilon t}}\Big)\Big]
    \end{aligned}
\end{equation*}
\begin{equation*}
    \begin{split}
        L^{\epsilon}(x,t)&=\frac{-1}{2\sqrt{\pi \epsilon t}}\int_0^\infty \Big\{e^{-\frac{(x-\xi)^2}{4\epsilon t}}-e^{-\frac{(x+\xi)^2}{4\epsilon t}}\Big\}\theta^{\epsilon}_0(-\xi)d\xi+\int_0^tg\prime(\tau)~~ erfc\Big(\frac{-x}{2\sqrt{\epsilon(t-\tau)}}\Big)d\tau\\&+g(0)erfc\Big(\frac{-x}{2\sqrt{\epsilon t}}\Big).
    \end{split}
\end{equation*}
Here, the terms $\theta^{\epsilon}_0$ and $g$ are as follows: 
\begin{equation}\label{definition 1}
    \begin{split}
        &\theta^{\epsilon}_0(x)=e^{-\frac{1}{2\epsilon}\int_{0}^xu_0(y)dy}\\
        &g(t)=g(0)e^{-\frac{t}{4\epsilon}}I_0(-\frac{t}{4\epsilon})+\frac{\epsilon^{\frac{3}{2}}}{\pi}\int_0^t\frac{1-e^{-\frac{\tau}{2\epsilon}}}{\tau^{\frac{3}{2}}}f(t-\tau)d\tau
        \end{split}
\end{equation}
        and $I_0$ is the Modified Bessel function of first kind. In addition, the term $f(t)$ is as follows:

    \begin{equation}\label{definition 2}
        f(t)=\frac{1}{2(\epsilon t)^{\frac{3}{2}}}\int_0^{\infty}\Big[e^{-\frac{t}{2\epsilon}}\theta^{\epsilon}_0(\xi)+\theta^{\epsilon}_0(-\xi)\Big]\xi e^{-\frac{\xi^2}{4\epsilon t}}d\xi-\frac{g(0)}{\sqrt{\epsilon t}}(e^{-\frac{t}{2\epsilon}}+1).
    \end{equation}

    Further $u^{\epsilon} \in C((0, \infty); H^1(\mathbb{R})\cap L^{\infty}(\mathbb{R}))\cap C^{\infty} \big(\mathbb{R}\setminus \{0\}) \times (0, \infty)\big)$ and satisfies the following weak formulation.
    \begin{equation}
    \label{weak formulation viscous problem}
    \int_{0}^\infty \int_{\mathbb{R}}\big[u^{\epsilon} \phi_t +\frac{1}{2} (u^{\epsilon})^2 \phi_x - \epsilon u^{\epsilon}_x \phi_x\big] dx dt +\int_{0}^{\infty} \phi(0,t) dt -\int_{\mathbb{R}} u_0 (x) \phi(x,0) dx=0,
    \end{equation}
    for all $\phi \in C_c ^{\infty} (\mathbb{R} \times [0, \infty)).$
\end{theorem}

\begin{proof}
Consider the initial value problem: \begin{equation}\label{IVP}
   \begin{split}
       u_t+uu_x-\epsilon u_{xx}&=\delta,~~~~~~~x\in\mathbb{R}, t>0\\
       u(x,0)&=u_0(x),~~x\in\mathbb{R}.
   \end{split}
\end{equation}
To convert this nonlinear equation to linear, we consider the following well-known Hopf-Cole  transformation:
\begin{equation*}
    \theta^{\epsilon}(x,t)=e^{-\frac{1}{2\epsilon}\int_{0}^{x}u(y,t)dy}.
\end{equation*}
Then \begin{equation*}
    \theta^{\epsilon}_x(x,t)=e^{-\frac{1}{2\epsilon}\int_{0}^{x}u(y,t)dy}\Big(-\frac{1}{2\epsilon}u(x,t)\Big).
\end{equation*}
Hence \begin{equation*}
    u(x,t)=-\frac{2\epsilon \theta^{\epsilon}_x}{\theta^{\epsilon}}
\end{equation*}
and \begin{equation*}
    \begin{split}
        \theta^{\epsilon}(x,0)=\theta^{\epsilon}_0(x)&=e^{-\frac{1}{2\epsilon}\int_{0}^{x}u(y,0)dy}\\
        &=e^{-\frac{1}{2\epsilon}\int_{0}^{x}u_0(y)dy}.
    \end{split}
\end{equation*}
Now, \begin{equation*}
        u_t=-2\epsilon\Big(\frac{\theta^{\epsilon}_t}{\theta^{\epsilon}}\Big)_x,
\end{equation*}
and \begin{equation*}
    \begin{split}
        &u_x=-2\epsilon\Big(\frac{\theta^{\epsilon}_x}{\theta^{\epsilon}}\Big)_x\\
        &u_{xx}=-2\epsilon\Big(\frac{\theta^{\epsilon}_x}{\theta^{\epsilon}}\Big)_{xx}.
    \end{split}
\end{equation*}
Then putting the above expressions in \eqref{IVP}, and simplifying, we get \begin{equation*}
        -2\epsilon\Big(\frac{\theta^{\epsilon}_t}{\theta^{\epsilon}}\Big)_x=\epsilon^2\Big[-2\Big(\frac{\theta^{\epsilon}_x}{\theta^{\epsilon}}\Big)_x-2\Big(\frac{\theta^{\epsilon}_x}{\theta^{\epsilon}}\Big)^2\Big]_x+\delta.
\end{equation*}
The above is true if $\theta^{\epsilon}$ satisfies 

~~~\begin{equation*}
       \theta^{\epsilon}_t=\epsilon\theta^{\epsilon}_{xx}-H(x)\frac{\theta^{\epsilon}}{2\epsilon},
\end{equation*} where $H(x)$ is the Heaviside function.
Therefore, the problem \eqref{IVP} reduces to \begin{equation}\label{reduced IVP}
    \begin{split}
          \theta^{\epsilon}_t&=\epsilon\theta^{\epsilon}_{xx}-H(x)\frac{\theta^{\epsilon}}{2\epsilon}~,x\in\mathbb{R}, t>0\\
        \theta^{\epsilon}(x,0)&=\theta^{\epsilon}_0(x)~,x\in\mathbb{R}.
    \end{split}
\end{equation}
Let \begin{math}
    R^{\epsilon}(x,t)
\end{math} be a solution in the right-half domain \begin{math}
    \{x> 0\}
\end{math} which satisfies \begin{equation}\label{IBVP 1}
   \begin{split}
       R^{\epsilon}_t-\epsilon R^{\epsilon}_{xx}&=-\frac{R^{\epsilon}}{2\epsilon},~~~~~x> 0, t>0\\
       R^{\epsilon}(x,0)&=\theta^{\epsilon}_0(x), ~~~~~x> 0\\
       R^{\epsilon}(0,t)&=g(t),~~~~~t>0.
   \end{split}
\end{equation}
Let \begin{math}
    L^{\epsilon}(x,t)
\end{math} be a solution in the left-half domain \begin{math}
    \{x<0\}
\end{math} which satisfies \begin{equation}\label{IBVP 2}
   \begin{split}
       L^{\epsilon}_t&=\epsilon L^{\epsilon}_{xx},~~~~~x< 0, t>0\\
       L^{\epsilon}(x,0)&=\theta^{\epsilon}_0(x), ~~~~~x< 0\\
       L^{\epsilon}(0,t)&=g(t),~~~~~t>0
   \end{split}
\end{equation}
where \begin{math}
    g(t)
\end{math} is a function of \begin{math}
    t
\end{math} which we shall determine.
Consider the transformation \begin{equation}\label{transformation 1}
     R^{\epsilon}(x,t)=e^{-\frac{t}{2\epsilon}}w^{\epsilon}(x,t).
\end{equation}
Then \begin{equation*}
    \begin{split}
       &R^{\epsilon}_t=e^{-\frac{t}{2\epsilon}}w^{\epsilon}_t+e^{-\frac{t}{2\epsilon}}(-\frac{1}{2\epsilon})w^{\epsilon},\\
       &R^{\epsilon}_{xx}=e^{-\frac{t}{2\epsilon}}w^{\epsilon}_{xx}.
    \end{split}
\end{equation*}
Then, simplifying \eqref{IBVP 1}, we have the following.
\begin{equation}\label{reduced IBVP 1}
    \begin{split}
        w^{\epsilon}_t&=\epsilon w^{\epsilon}_{xx},~~~~~x> 0, t>0\\
        w^{\epsilon}(x,0)&=\theta^{\epsilon}_0(x),~~~~~x>0\\
        w^{\epsilon}(0,t)&=e^{\frac{t}{2\epsilon}}g(t),~~~ t>0.
    \end{split}
\end{equation}
Therefore the solution of \eqref{reduced IBVP 1} is given by 
\begin{equation}\label{definition 3}
        w^{\epsilon}(x,t)=\frac{1}{2\sqrt{\pi \epsilon t}}\int_0^\infty \Big\{e^{-\frac{(x-\xi)^2}{4\epsilon t}}-e^{-\frac{(x+\xi)^2}{4\epsilon t}}\Big\}\theta^{\epsilon}_0(\xi)d\xi+\frac{x}{2\sqrt{\pi \epsilon}}\int_0^te^{-\frac{x^2}{4 \epsilon (t-\tau)}}\frac{e^{\frac{\tau}{2\epsilon}}g(\tau)}{(t-\tau)^{\frac{3}{2}}}d\tau.
\end{equation}
Now, the error function and complementary error function are defined, respectively. \begin{equation*}
    \begin{split}
        &err(z)=\frac{2}{\sqrt{\pi}}\int_0^z e^{-u^2}du,\\
 &erfc(z)=1-err(z)=1-\frac{2}{\sqrt{\pi}}\int_0^z e^{-u^2}du.
    \end{split}
\end{equation*}
To simplify the second term in \eqref{definition 3}, consider the integral \begin{equation*}
    \begin{split}
        &\int_0^t\Big(e^{\frac{\tau}{2\epsilon}}g(\tau)\Big)^\prime~~ erfc\Big(\frac{x}{2\sqrt{\epsilon(t-\tau)}}\Big)d\tau\\=&\Big[e^{\frac{\tau}{2\epsilon}}g(\tau)erfc\Big(\frac{x}{2\sqrt{\epsilon(t-\tau)}}\Big)\Big]_0^t-\int_0^te^{\frac{\tau}{2\epsilon}}g(\tau)(-\frac{2}{\sqrt{\pi}})e^{-\frac{x^2}{4 \epsilon (t-\tau)}}\frac{(-x)(-1)}{4\sqrt{\epsilon }(t-\tau)^{\frac{3}{2}}}d\tau\\=&-g(0)erfc\Big(\frac{x}{2\sqrt{\epsilon t}}\Big)+\int_0^t \frac{x}{2\sqrt{\pi \epsilon}}e^{-\frac{x^2}{4 \epsilon (t-\tau)}}\frac{e^{\frac{\tau}{2\epsilon}}g(\tau)}{(t-\tau)^{\frac{3}{2}}}d\tau.\\
    \end{split}
\end{equation*}
So,
\begin{equation*}
    \begin{split}
         \frac{x}{2\sqrt{\pi \epsilon}}\int_0^t e^{-\frac{x^2}{4 \epsilon (t-\tau)}}\frac{e^{\frac{\tau}{2\epsilon}}g(\tau)}{(t-\tau)^{\frac{3}{2}}}d\tau&=\int_0^t\Big(e^{\frac{\tau}{2\epsilon}}g(\tau)\Big)^\prime~~ erfc\Big(\frac{x}{2\sqrt{\epsilon(t-\tau)}}\Big)d\tau\\&+g(0)erfc\Big(\frac{x}{2\sqrt{\epsilon t}}\Big).
    \end{split}
\end{equation*}
Then from \eqref{definition 3} we get, \begin{equation*}
    \begin{split}
    w^{\epsilon}(x,t)&=\frac{1}{2\sqrt{\pi \epsilon t}}\int_0^\infty \Big\{e^{-\frac{(x-\xi)^2}{4\epsilon t}}-e^{-\frac{(x+\xi)^2}{4\epsilon t}}\Big\}\theta^{\epsilon}_0(\xi)d\xi+\int_0^t\Big(e^{\frac{\tau}{2\epsilon}}g(\tau)\Big)^\prime~~ erfc\Big(\frac{x}{2\sqrt{\epsilon(t-\tau)}}\Big)d\tau\\&+g(0)erfc\Big(\frac{x}{2\sqrt{\epsilon t}}\Big).
    \end{split}
\end{equation*}
Therefore, from \eqref{transformation 1} we have \begin{equation*}
    \begin{split}
        R^{\epsilon}(x,t)&=e^{-\frac{t}{2\epsilon}}\Big[\frac{1}{2\sqrt{\pi \epsilon t}}\int_0^\infty \Big\{e^{-\frac{(x-\xi)^2}{4\epsilon t}}-e^{-\frac{(x+\xi)^2}{4\epsilon t}}\Big\}\theta^{\epsilon}_0(\xi)d\xi\\&+\int_0^t\Big(e^{\frac{\tau}{2\epsilon}}g(\tau)\Big)^\prime~~ erfc\Big(\frac{x}{2\sqrt{\epsilon(t-\tau)}}\Big)d\tau+g(0)erfc\Big(\frac{x}{2\sqrt{\epsilon t}}\Big)\Big].
    \end{split}
\end{equation*}
Now, the solution of \eqref{IBVP 2} is given by \begin{equation}\label{definition 4}
    \begin{split}
        &L^{\epsilon}(x,t)\\&=\frac{-1}{2\sqrt{\pi \epsilon t}}\int_0^\infty \Big\{e^{-\frac{(x-\xi)^2}{4\epsilon t}}-e^{-\frac{(x+\xi)^2}{4\epsilon t}}\Big\}\theta^{\epsilon}_0(-\xi)d\xi-\frac{x}{2\sqrt{\pi \epsilon}}\int_0^te^{-\frac{x^2}{4 \epsilon (t-\tau)}}\frac{g(\tau)}{(t-\tau)^{\frac{3}{2}}}d\tau.
    \end{split}
\end{equation}
To simplify the second term in \eqref{definition 4}, consider the integral 
\begin{equation*}
    \begin{split}
        &\int_0^tg\prime(\tau)~~ erfc\Big(\frac{-x}{2\sqrt{\epsilon(t-\tau)}}\Big)d\tau\\&=\Big[g(\tau)erfc\Big(\frac{-x}{2\sqrt{\epsilon(t-\tau)}}\Big)\Big]_0^t-\int_0^tg(\tau)(-\frac{2}{\sqrt{\pi}})e^{-\frac{x^2}{4 \epsilon (t-\tau)}}\frac{(-x)(-1)(-1)}{4\sqrt{\epsilon }(t-\tau)^{\frac{3}{2}}}d\tau\\&=-g(0)erfc\Big(\frac{-x}{2\sqrt{\epsilon t}}\Big)-\int_0^t \frac{x}{2\sqrt{\pi \epsilon}}e^{-\frac{x^2}{4 \epsilon (t-\tau)}}\frac{g(\tau)}{(t-\tau)^{\frac{3}{2}}}d\tau.
    \end{split}
\end{equation*}
So,
\begin{equation*}
    \begin{split}
         \frac{-x}{2\sqrt{\pi \epsilon}}\int_0^t e^{-\frac{x^2}{4 \epsilon (t-\tau)}}\frac{g(\tau)}{(t-\tau)^{\frac{3}{2}}}d\tau&=\int_0^tg\prime(\tau)~~ erfc\Big(\frac{-x}{2\sqrt{\epsilon(t-\tau)}}\Big)d\tau\\&+g(0)erfc\Big(\frac{-x}{2\sqrt{\epsilon t}}\Big).
    \end{split}
\end{equation*}
Therefore, from \eqref{definition 4} we have,\begin{equation}\label{9}
    \begin{split}
        L^{\epsilon}(x,t)&=\frac{-1}{2\sqrt{\pi \epsilon t}}\int_0^\infty \Big\{e^{-\frac{(x-\xi)^2}{4\epsilon t}}-e^{-\frac{(x+\xi)^2}{4\epsilon t}}\Big\}\theta^{\epsilon}_0(-\xi)d\xi+\int_0^tg^\prime(\tau)~~ erfc\Big(\frac{-x}{2\sqrt{\epsilon(t-\tau)}}\Big)d\tau\\&+g(0)erfc\Big(\frac{-x}{2\sqrt{\epsilon t}}\Big).
    \end{split}
\end{equation}
Now we calculate \begin{math}
    R^{\epsilon}_x(x,t) ~~\text{and}~~ L^{\epsilon}_x(x,t) 
\end{math}~and then we assume \begin{math}
     R^{\epsilon}_x(0,t)=L^{\epsilon}_x(0,t)
\end{math} to find the value of the unknown function \begin{math}
    g(t).
\end{math}
Now, \begin{equation}\label{definition 5}
    \begin{split}
        R^{\epsilon}_x(x,t)&=e^{-\frac{t}{2\epsilon}}\Big[\frac{1}{2\sqrt{\pi \epsilon t}}\int_0^\infty \Big\{e^{-\frac{(x-\xi)^2}{4\epsilon t}}\frac{-(x-\xi)}{2\epsilon t}-e^{-\frac{(x+\xi)^2}{4\epsilon t}}\frac{-(x+\xi)}{2\epsilon t}\Big\}\theta^{\epsilon}_0(\xi)d\xi\\&+\int_0^t\Big(e^{\frac{\tau}{2\epsilon}}g(\tau)\Big)^\prime~~\Big(-\frac{2}{\sqrt{\pi}}\Big)\frac{e^{\frac{-x^2}{4\epsilon(t-\tau)} }}{2\sqrt{\epsilon(t-\tau)}}d\tau+g(0)\Big(-\frac{2}{\sqrt{\pi}}\Big)\frac{e^{\frac{-x^2}{4\epsilon t} }}{2\sqrt{\epsilon t)}}\Big]\\&=e^{-\frac{t}{2\epsilon}}\Big[\frac{1}{2\sqrt{\pi \epsilon t}}\int_0^\infty \Big\{\frac{(\xi-x)}{2\epsilon t}e^{-\frac{(x-\xi)^2}{4\epsilon t}}+\frac{(x+\xi)}{2\epsilon t}e^{-\frac{(x+\xi)^2}{4\epsilon t}}\Big\}\theta^{\epsilon}_0(\xi)d\xi\\&-\int_0^t\Big(e^{\frac{\tau}{2\epsilon}}g(\tau)\Big)^\prime~~\frac{e^{\frac{-x^2}{4\epsilon(t-\tau)} }}{\sqrt{\pi\epsilon(t-\tau)}}d\tau-\frac{g(0)e^{\frac{-x^2}{4\epsilon t} }}{\sqrt{\pi\epsilon t}}\Big]\\&=e^{-\frac{t}{2\epsilon}}\Big[\frac{1}{4\sqrt{\pi}(\epsilon t)^{\frac{3}{2}}}\int_0^{\infty} \theta^{\epsilon}_0(\xi)\Big((\xi-x)e^{-\frac{(x-\xi)^2}{4\epsilon t}}+(\xi+x)e^{-\frac{(x+\xi)^2}{4\epsilon t}}\Big)d\xi\Big]\\&-\frac{e^{-\frac{t}{2\epsilon}}}{\sqrt{\pi}}\int_0^t\Big(e^{\frac{\tau}{2\epsilon}}g(\tau)\Big)^\prime \frac{e^{\frac{-x^2}{4\epsilon(t-\tau)} }}{\sqrt{\epsilon(t-\tau)}}d\tau-\frac{e^{-\frac{t}{2\epsilon}}g(0)}{\sqrt{\pi}}\frac{e^{\frac{-x^2}{4\epsilon t} }}{\sqrt{\epsilon t}}
    \end{split}
\end{equation}
and, \begin{equation}\label{definition 6}
    \begin{split}
        L^{\epsilon}_x(x,t)&=\frac{-1}{2\sqrt{\pi \epsilon t}}\int_0^\infty \Big\{e^{-\frac{(x-\xi)^2}{4\epsilon t}}\frac{-(x-\xi)}{2\epsilon t}-e^{-\frac{(x+\xi)^2}{4\epsilon t}}\frac{-(x+\xi)}{2\epsilon t}\Big\}\theta^{\epsilon}_0(-\xi)d\xi\\&+\int_0^tg\prime(\tau)~~\Big(-\frac{2}{\sqrt{\pi}}\Big)\frac{-e^{\frac{-x^2}{4\epsilon(t-\tau)} }}{2\sqrt{\epsilon(t-\tau)}}d\tau+g(0)\Big(-\frac{2}{\sqrt{\pi}}\Big)\frac{-e^{\frac{-x^2}{4\epsilon t} }}{2\sqrt{\epsilon t)}}\\&=\frac{-1}{2\sqrt{\pi \epsilon t}}\int_0^\infty \Big\{e^{-\frac{(x-\xi)^2}{4\epsilon t}}\frac{-(x-\xi)}{2\epsilon t}+e^{-\frac{(x+\xi)^2}{4\epsilon t}}\frac{(x+\xi)}{2\epsilon t}\Big\}\theta^{\epsilon}_0(-\xi)d\xi\\&+\int_0^t\frac{g^\prime(\tau)}{\sqrt{\pi\epsilon(t-\tau)}}e^{\frac{-x^2}{4\epsilon(t-\tau)} }d\tau+\frac{g(0)}{\sqrt{\pi\epsilon t}}e^{\frac{-x^2}{4\epsilon t} }\\&=\frac{-1}{4\sqrt{\pi}(\epsilon t)^{\frac{3}{2}}}\int_0^{\infty}\theta^{\epsilon}_0(-\xi)\Big\{(\xi-x)e^{-\frac{(x-\xi)^2}{4\epsilon t}}+(x+\xi)e^{-\frac{(x+\xi)^2}{4\epsilon t}}\Big\}d\xi\\&+\int_0^t\frac{g^\prime(\tau)}{\sqrt{\pi\epsilon(t-\tau)}}e^{\frac{-x^2}{4\epsilon(t-\tau)} }d\tau+\frac{g(0)}{\sqrt{\pi\epsilon t}}e^{\frac{-x^2}{4\epsilon t}}.
    \end{split}
\end{equation}
Putting \begin{math}
    x=0
\end{math} in \eqref{definition 5} and \eqref{definition 6} we get the following respectively, \begin{equation*}
\begin{split}
     R^{\epsilon}_x(0,t)&=e^{-\frac{t}{2\epsilon}}\Big[\frac{1}{4\sqrt{\pi}(\epsilon t)^{\frac{3}{2}}}\int_0^{\infty}2\theta^{\epsilon}_0(\xi)\xi e^{-\frac{\xi^2}{4\epsilon t}}d\xi\Big]\\&-\frac{e^{-\frac{t}{2\epsilon}}}{\sqrt{\pi}}\int_0^t\Big(e^{\frac{\tau}{2\epsilon}}g(\tau)\Big)^\prime \frac{1}{\sqrt{\epsilon(t-\tau)}}d\tau-\frac{e^{-\frac{t}{2\epsilon}}g(0)}{\sqrt{\pi\epsilon t}},
\end{split}
\end{equation*}
and \begin{equation*}
    \begin{split}
        L^{\epsilon}_x(0,t)&=\frac{-1}{4\sqrt{\pi}(\epsilon t)^{\frac{3}{2}}}\int_0^{\infty}\theta^{\epsilon}_0(-\xi)\{2\xi e^{\frac{-\xi^2}{4\epsilon t}}\}d\xi+\int_0^tg^\prime(\tau)\frac{1}{\sqrt{\pi \epsilon (t-\tau)}}d\tau+\frac{g(0)}{\sqrt{\pi\epsilon t}}.
    \end{split}
\end{equation*}
Assuming \begin{math}
    R^{\epsilon}_x(0,t)=L^{\epsilon}_x(0,t),
\end{math} we get, \begin{equation*}
    \begin{split}
        &\frac{1}{2\sqrt{\pi}(\epsilon t)^{\frac{3}{2}}}\int_0^{\infty}e^{-\frac{t}{2\epsilon}}\theta^{\epsilon}_0(\xi)\xi e^{-\frac{\xi^2}{4\epsilon t}}d\xi-\frac{e^{-\frac{t}{2\epsilon}}}{\sqrt{\pi}}\int_0^t\Big(e^{\frac{\tau}{2\epsilon}}g(\tau)\Big)^\prime \frac{1}{\sqrt{\epsilon(t-\tau)}}d\tau-\frac{e^{-\frac{t}{2\epsilon}}g(0)}{\sqrt{\pi\epsilon t}}\\&=\frac{-1}{2\sqrt{\pi}(\epsilon t)^{\frac{3}{2}}}\int_0^{\infty}\theta^{\epsilon}_0(-\xi)\xi e^{\frac{-\xi^2}{4\epsilon t}}d\xi+\frac{1}{\sqrt{\pi}}\int_0^tg^\prime(\tau)\frac{1}{\sqrt{ \epsilon (t-\tau)}}d\tau+\frac{g(0)}{\sqrt{\pi\epsilon t}}.
    \end{split}
\end{equation*}
Simplifying this equation, we get \begin{equation*}
    \begin{split}
        &\frac{1}{2(\epsilon t)^{\frac{3}{2}}}\int_0^{\infty}\Big[e^{-\frac{t}{2\epsilon}}\theta^{\epsilon}_0(\xi)+\theta^{\epsilon}_0(-\xi)\Big]\xi e^{-\frac{\xi^2}{4\epsilon t}}d\xi-\frac{g(0)}{\sqrt{\epsilon t}}(e^{-\frac{t}{2\epsilon}}+1)\\&=\int_0^te^{-\frac{(t-\tau)}{2\epsilon}}\Big(\frac{g(\tau)}{2\epsilon}+g^\prime(\tau)\Big)\frac{1}{\sqrt{ \epsilon (t-\tau)}}d\tau+\int_0^t\frac{g^\prime(\tau)}{\sqrt{\epsilon(t-\tau)}}d\tau
    \end{split}
\end{equation*}
which decides the boundary condition \begin{math}
    g(t).
\end{math} In the above equation, let us denote the LHS as $  f(t).$
Then the above equation can be written as \begin{equation}\label{definition 7}
    \begin{split}
        f(t)&=\frac{1}{2(\epsilon t)^{\frac{3}{2}}}\int_0^{\infty}\Big[e^{-\frac{t}{2\epsilon}}\theta^{\epsilon}_0(\xi)+\theta^{\epsilon}_0(-\xi)\Big]\xi e^{-\frac{\xi^2}{4\epsilon t}}d\xi-\frac{g(0)}{\sqrt{\epsilon t}}(e^{-\frac{t}{2\epsilon}}+1)\\&=\Big[\Big(\frac{g(\tau)}{2\epsilon}+g^\prime(\tau)\Big)*\frac{e^{-\frac{\tau}{2\epsilon}}}{\sqrt{\epsilon \tau}}\Big](t)+\Big[g^\prime(\tau)*\frac{1}{\sqrt{\epsilon \tau}}\Big](t).
    \end{split}
\end{equation}
Now taking the Laplace transform on both sides of \eqref{definition 7} we get \begin{equation*}
   \begin{split}
        L\{f(t)\}&=L\Big\{\Big(\frac{g(\tau)}{2\epsilon}+g^\prime(\tau)\Big)\Big\}L\Big\{\frac{e^{-\frac{\tau}{2\epsilon}}}{\sqrt{\epsilon \tau}}\Big\}+L\Big\{g^\prime(\tau)\Big\}L\Big\{\frac{1}{\sqrt{\epsilon \tau}}\Big\}\\&=\Big[\frac{G(s)}{2\epsilon}+sG(s)-g(0)\Big]\frac{\sqrt{\pi}}{\sqrt{\epsilon(s+\frac{1}{2\epsilon})}}+\Big[sG(s)-g(0)\Big]\frac{\sqrt{\pi}}{\sqrt{\epsilon s}}
   \end{split}
\end{equation*} where $G(s)$ is the Laplace transform of $g(\tau)$.
Simplifying this equation, we get \begin{equation*}
    G(s)=\frac{g(0)}{\sqrt{s(s+\frac{1}{2\epsilon})}}+\frac{\sqrt{\epsilon} L\{f(t)\}}{\sqrt{\pi}\Big(\sqrt{s}+\sqrt{s+\frac{1}{2\epsilon}}\Big)}.
\end{equation*}
Taking the inverse Laplace transform, we get the following 
\begin{equation}\label{defnition 8}
    g(t)=L^{-1}\Big\{\frac{g(0)}{\sqrt{s(s+\frac{1}{2\epsilon})}}\Big\}+\frac{\sqrt{\epsilon}}{\sqrt{\pi}}L^{-1}\Big\{\frac{L\{f(t)\}}{\sqrt{s}+\sqrt{s+\frac{1}{2\epsilon}}}\Big\}.
\end{equation}
Now we have for any constants $a,b,$ \begin{equation*}
    \begin{split}
        &L^{-1}\Big\{\frac{1}{\sqrt{(s+a)(s+b)}}\Big\}=e^{-\frac{(a+b)}{2}t}I_0(\frac{(a-b)t}{2})\\&L^{-1}\Big\{\frac{1}{\sqrt{(s+a)}+\sqrt{(s+b)}}\Big\}=\frac{1}{a-b}\Big(\frac{e^{-bt}-e^{-at}}{2\sqrt{\pi t^3}}\Big)
    \end{split}
\end{equation*}
where $I_0$ is  the Modified Bessel function of the first kind and for any functions $P,Q$ $$L^{-1}\Big\{L\{P(t)\}L\{Q(t)\}\Big\}=(P*Q)(t).$$
Hence from equation  \eqref{defnition 8}
\begin{equation*}
   g(t)=g(0)e^{-\frac{t}{4\epsilon}}I_0(-\frac{t}{4\epsilon})+\frac{\epsilon^{\frac{3}{2}}}{\pi}\int_0^t\frac{1-e^{-\frac{\tau}{2\epsilon}}}{\tau^{\frac{3}{2}}}f(t-\tau)d\tau.
\end{equation*}
Therefore, in the right-half domain \begin{math}
    \{x>0\},
\end{math}
\begin{math}
    u^{\epsilon}(x,t)
\end{math}
\begin{equation*}
       =\frac{\begin{aligned}
       &-\frac{1}{2t^{\frac{3}{2}}\sqrt{\pi \epsilon}}\int_0^{\infty}\theta^{\epsilon}_0(\xi)\Big[(\xi-x)e^{-\frac{(x-\xi)^2}{4\epsilon t}}+(\xi+x)e^{-\frac{(x+\xi)^2}{4\epsilon t}}\Big]d\xi+2\sqrt{\frac{\epsilon}{\pi}}\int_0^t(e^{\frac{\tau}{2\epsilon}}g(\tau))^\prime\frac{e^{-\frac{x^2}{4\epsilon(t-\tau)}}}{\sqrt{t-\tau}}d\tau\\&+2\sqrt{\frac{\epsilon}{\pi}}g(0)\frac{e^{-\frac{x^2}{4\epsilon t}}}{\sqrt{t}}\end{aligned}}{\frac{1}{2\sqrt{\pi \epsilon t}}\int_0^{\infty}\Big[e^{-\frac{(x-\xi)^2}{4\epsilon t}}-e^{-\frac{(x+\xi)^2}{4\epsilon t}}\Big]\theta^{\epsilon}_0(\xi)d\xi+\int_0^t(e^{\frac{\tau}{2\epsilon}}g(\tau))^\prime ~erfc\Big(\frac{x}{2\sqrt{\epsilon(t-\tau)}}\Big)d\tau+g(0)erfc\Big(\frac{x}{2\sqrt{\epsilon t}}\Big)}
\end{equation*}
and in the left-half domain \begin{math}
    \{x<0\},
\end{math}
\begin{math}
    u^{\epsilon}(x,t)
\end{math}
\begin{equation*}
       =\frac{\begin{aligned}
       &\frac{1}{2t^{\frac{3}{2}}\sqrt{\pi \epsilon}}\int_0^{\infty}\theta^{\epsilon}_0(-\xi)\Big[(\xi-x)e^{-\frac{(x-\xi)^2}{4\epsilon t}}+(\xi+x)e^{-\frac{(x+\xi)^2}{4\epsilon t}}\Big]d\xi-2\sqrt{\frac{\epsilon}{\pi}}\int_0^tg^\prime(\tau)\frac{e^{-\frac{x^2}{4\epsilon(t-\tau)}}}{\sqrt{t-\tau}}d\tau\\&-2\sqrt{\frac{\epsilon}{\pi}}g(0)\frac{e^{-\frac{x^2}{4\epsilon t}}}{\sqrt{t}}\end{aligned}}{\frac{-1}{2\sqrt{\pi \epsilon t}}\int_0^{\infty}\Big[e^{-\frac{(x-\xi)^2}{4\epsilon t}}-e^{-\frac{(x+\xi)^2}{4\epsilon t}}\Big]\theta^{\epsilon}_0(-\xi)d\xi+\int_0^tg^\prime(\tau) ~erfc\Big(\frac{-x}{2\sqrt{\epsilon(t-\tau)}}\Big)d\tau+g(0)erfc\Big(\frac{-x}{2\sqrt{\epsilon t}}\Big)}
\end{equation*}
where 
\begin{equation*}
    \begin{split}
        &\theta^{\epsilon}_0(x)=e^{-\frac{1}{2\epsilon}\int_{0}^xu_0(y)dy}\\
        &g(t)=g(0)e^{-\frac{t}{4\epsilon}}I_0(-\frac{t}{4\epsilon})+\frac{\epsilon^{\frac{3}{2}}}{\pi}\int_0^t\frac{1-e^{-\frac{\tau}{2\epsilon}}}{\tau^{\frac{3}{2}}}f(t-\tau)
        d\tau\\&f(t)=\frac{1}{2(\epsilon t)^{\frac{3}{2}}}\int_0^{\infty}\Big[e^{-\frac{t}{2\epsilon}}\theta^{\epsilon}_0(\xi)+\theta^{\epsilon}_0(-\xi)\Big]\xi e^{-\frac{\xi^2}{4\epsilon t}}d\xi-\frac{g(0)}{\sqrt{\epsilon t}}(e^{-\frac{t}{2\epsilon}}+1).
\end{split}
\end{equation*}
 We can verify using the same method as in \cite{chung} that $u^{\epsilon} \in C((0, \infty); H^1(\mathbb{R})\cap L^{\infty}(\mathbb{R}))$ and satisfies \eqref{weak formulation viscous problem}. The calculations above show that the solution is unique and $u^{\epsilon}$ is $C^{\infty}$ away from the boundary.
\end{proof}
\section{Vanishing viscosity limit} 
In this section, we study the limiting behavior of the solutions of viscous regularizations. The limit can be formulated as the minimization of three functionals. This fact is described in the following theorem.
\begin{theorem}\label{vanishing viscosity limit}
For a.e. $(x,t)\in (0,\infty)\times(0,\infty) $, the approximations $[-2\epsilon \log R^{\epsilon}(x,t)]$ converge to $U_R(x,t),$
where $$U_R(x,t)=\min\Big (U_R^ 1(x,t), U_R^ 2(x,t),U_R^3(x,t)\Big)+t$$ 
and $U_R^ 1(x,t)$, $U_R^ 2(x,t)$, $U_R^3(x,t)$ are given by
\begin{equation*}
\begin{aligned}
U_R^ 1(x, t)&=\min_{\substack{0 \le \tau < t \\ 0 \le u < \tau \\ \xi \ge 0}}\Big[ \frac{x^2}{2(t-\tau)}+\frac{\xi ^2}{2 (\tau-u)}-u +\int_{0}^{\xi} u_0(z)dz \Big]\\
U_R^ 2(x,t)&=\min_{\substack{0 \le \tau < t  \\ \xi \ge 0}}\Big[ \frac{x^2}{2(t-\tau)}+\frac{\xi ^2}{2 \tau} -\tau+\int_{0}^{-\xi} u_0(z)dz \Big]\\
U_R^3(x,t)&= \min_{\xi \geq 0}\Big[ \frac{(x-\xi)^2}{2 t} + \int_{0}^{\xi} u_0(z)dz\Big].
\end{aligned}
\end{equation*}
For a.e. $(x,t)\in (-\infty,0)\times(0,\infty)$, the approximations $[-2\epsilon \log L^{\epsilon}(x,t)]$ converge to $U_L(x,t),$
where $$U_L(x,t)=\min\Big (U_L^ 1(x,t), U_L^ 2(x,t),U_L^3(x,t)\Big)$$ 
and $U_L^ 1(x,t)$, $U_L^ 2(x,t)$, $U_L^3(x,t)$ are given by
\begin{equation*}
\begin{aligned}
U_L^ 1(x, t)&=\min_{\substack{0 \le \tau < t \\ 0 \le u < \tau \\ \xi \ge 0}}\Big[ \frac{x^2}{2(t-\tau)}+\frac{\xi ^2}{2 (\tau-u)}+\tau-u +\int_{0}^{\xi} u_0(z)dz \Big]\\
U_L^ 2(x,t)&=\min_{\substack{0 \le \tau < t  \\ \xi \ge 0}}\Big[ \frac{x^2}{2(t-\tau)}+\frac{\xi ^2}{2 \tau}+\int_{0}^{-\xi} u_0(z)dz \Big]\\
U_L^3(x,t)&= \min_{\xi \geq 0}\Big[ \frac{(x+\xi)^2}{2 t} + \int_{0}^{-\xi} u_0(z)dz\Big].
\end{aligned}
\end{equation*}
\end{theorem}

The proof of the above theorem is structured in multiple steps and expressed through a sequence of lemmas.
\begin{lemma} \label{rewriting R}
For all $(x,t)\in (0,\infty)\times (0, \infty),$ $R^{\epsilon}(x,t)$ can be split as follows:
{%
\setlength{\abovedisplayskip}{8pt}
\setlength{\abovedisplayshortskip}{8pt}
\begin{equation}\label{definition 9}
  \begin{aligned}
   &R^{\epsilon}(x,t)=e^{-\frac{t}{2 \epsilon}} [I_1^{\epsilon} (x,t) -I_2^{\epsilon}(x,t)+I_3^{\epsilon}(x,t) +I_4^{\epsilon}(x,t)+I_5^{\epsilon}(x,t)-I_6^{\epsilon}(x,t)-I_7^{\epsilon}(x,t)]\\
\text{where},~~
   &I_1^{\epsilon}(x,t)= \frac{1}{2\sqrt{\pi \epsilon t}} \int_0^\infty e^{-\frac{1}{2 \epsilon}[\frac{(x-\xi)^2}{2t}+\int_{0}^\xi u_0 (y) dy]} d \xi\\
   &I_2^{\epsilon} (x,t)= \frac{1}{2\sqrt{\pi \epsilon t}} \int_0^\infty e^{-\frac{1}{2 \epsilon}[\frac{(x+\xi)^2}{2t}+\int_{0}^\xi u_0 (y) dy]} d \xi\\
   &I_3^{\epsilon}(x,t)= g(0) \frac{x}{2 \sqrt{\pi \epsilon}}\int_0^t I_0(-\frac{\tau}{4\epsilon})\frac{e^{-\frac{1}{4 \epsilon}\Big[ \frac{x^2}{t-\tau}-\tau\Big]}}{(t-\tau)^{\frac{3}{2}}} \, \mathrm{d}\tau\\
   &I_4^{\epsilon}(x,t)=  \frac{x}{(4 \pi \epsilon)^{\frac{3}{2}}} \int_0^t \int_0^\tau \int_0^1 \int_{0}^{\infty}
      \frac{\xi  e^{-\frac{1}{2 \epsilon}\Big[ \frac{x^2}{2(t-\tau)}+\frac{\xi ^2}{2(\tau-u)}+(\theta-1) u+\int_{0}^\xi u_0 (y) dy \Big]}}{u^{\frac{1}{2}}(\tau-u)^{\frac{3}{2}}(t-\tau)^{\frac{3}{2}}}
     d \xi d\theta du d\tau\\
   &I_5^{\epsilon}(x,t)=\frac{x}{(4 \pi \epsilon)^{\frac{3}{2}}} \int_0^t \int_0^\tau \int_0^1 \int_{0}^{\infty}
      \frac{\xi e^{-\frac{1}{2 \epsilon}\Big[ \frac{x^2}{2(t-\tau)}+\frac{\xi ^2}{2(\tau-u)}+\theta u-\tau +\int_{0}^{-\xi} u_0 (y) dy \Big]}}{ u^{\frac{1}{2}} (\tau-u)^{\frac{3}{2}}(t-\tau)^{\frac{3}{2}}}
      d \xi d\theta du d\tau\\
   &I_6^{\epsilon}(x,t)= \frac{x g(0)}{4 \pi^{\frac{3}{2}} \epsilon^{\frac{1}{2}}}  \int_0^t \int_0^\tau \int_0^1  \frac{1}{ u^{\frac{1}{2}} (\tau-u)^{\frac{1}{2}}(t-\tau)^{\frac{3}{2}}}e^{-\frac{1}{2 \epsilon}\Big[ \frac{x^2}{2(t-\tau)}+(\theta-1) u\Big]} d\theta du d\tau\\
     &I_7^{\epsilon}(x,t)=\frac{x g(0)}{4 \pi^{\frac{3}{2}} \epsilon^{\frac{1}{2}}}  \int_0^t \int_0^\tau \int_0^1  \frac{1}{ u^{\frac{1}{2}} (\tau-u)^{\frac{1}{2}}(t-\tau)^{\frac{3}{2}}}e^{-\frac{1}{2 \epsilon}\Big[ \frac{x^2}{2(t-\tau)}-\tau +\theta u\Big]} d\theta du d\tau.\\
   \end{aligned}
   \end{equation}}
   \end{lemma}
\begin{proof}
  Using integration by parts, the expression for $R^{\epsilon}(x,t)$ can be simplified to 
  \begin{equation}
\label{expression for R}
    \begin{aligned}
        R^{\epsilon}(x,t)&=e^{-\frac{t}{2\epsilon}}\Big[\frac{1}{2\sqrt{\pi \epsilon t}}\int_0^\infty \Big\{e^{-\frac{(x-\xi)^2}{4\epsilon t}}-e^{-\frac{(x+\xi)^2}{4\epsilon t}}\Big\}\theta^{\epsilon}_0(\xi)d\xi\\
        &+ \frac{x}{2 \sqrt{\pi \epsilon}}\int_0^t \frac{g(\tau)}{(t-\tau)^{\frac{3}{2}}}e^{-\frac{1}{2 \epsilon}\Big[ \frac{x^2}{2(t-\tau)}-\tau\Big]} d\tau\Big]\\
        &=e^{-\frac{t}{2\epsilon}}[R_1^{\epsilon}(x,t) +R_2^{\epsilon}(x,t)].
        \end{aligned}
        \end{equation}
$R_1^{\epsilon}(x,t)$ can be rewritten as
\begin{equation}\label{expression for R1}
    \begin{aligned}
     R_1^{\epsilon}(x,t)&=  \frac{1}{2\sqrt{\pi \epsilon t}} \int_0^\infty \Big\{e^{-\frac{(x-\xi)^2}{4\epsilon t}}-e^{-\frac{(x+\xi)^2}{4\epsilon t}}\Big\}\theta^{\epsilon}_0(\xi)d\xi\\
     &= \frac{1}{2\sqrt{\pi \epsilon t}} \int_0^\infty e^{-\frac{1}{2 \epsilon}[\frac{(x-\xi)^2}{2t}+\int_{0}^\xi u_0 (y) dy]} d \xi-\frac{1}{2\sqrt{\pi \epsilon t}} \int_0^\infty e^{-\frac{1}{2 \epsilon}[\frac{(x+\xi)^2}{2t}+\int_{0}^\xi u_0 (y) dy]} d \xi\\
     &=I_1^{\epsilon}(x,t) -I_2^{\epsilon}(x,t).
    \end{aligned}
\end{equation}
Putting the value of $g$, we simplify $R_2^{\epsilon}(x,t)$ as follows:
\begin{equation}\label{expression for R2}
    \begin{aligned}
    R_2^{\epsilon}(x,t)&=   \frac{x}{2 \sqrt{\pi \epsilon}}\int_0^t \frac{g(\tau)}{(t-\tau)^{\frac{3}{2}}}e^{-\frac{1}{2 \epsilon}\Big[ \frac{x^2}{2(t-\tau)}-\tau\Big]} d\tau\\
    &=  \frac{x}{2 \sqrt{\pi \epsilon}}\int_0^t \frac{e^{-\frac{1}{2 \epsilon}\Big[ \frac{x^2}{2(t-\tau)}-\tau\Big]}}{(t-\tau)^{\frac{3}{2}}}
    \Big[g(0)e^{-\frac{\tau}{4\epsilon}}I_0(-\frac{\tau}{4\epsilon})+\frac{\epsilon^{\frac{3}{2}}}{\pi}\int_0^\tau\frac{1-e^{-\frac{u}{2\epsilon}}}{u^{\frac{3}{2}}}f(\tau-u)du \Big] d\tau \\
   &= g(0) \frac{x}{2 \sqrt{\pi \epsilon}}\int_0^t I_0(-\frac{\tau}{4\epsilon})\frac{e^{-\frac{1}{4 \epsilon}\Big[ \frac{x^2}{t-\tau}-\tau\Big]}}{(t-\tau)^{\frac{3}{2}}} \, \mathrm{d}\tau\\& +  
     \frac{x}{2 \sqrt{\pi \epsilon}}\frac{\epsilon^{\frac{3}{2}}}{\pi}\int_0^t \frac{e^{-\frac{1}{2 \epsilon}\Big[ \frac{x^2}{2(t-\tau)}-\tau\Big]}}{(t-\tau)^{\frac{3}{2}}} \int_0^\tau\frac{1-e^{-\frac{u}{2\epsilon}}}{u^{\frac{3}{2}}}f(\tau-u) \mathrm{d}u \mathrm{d}\tau\\
    &= I_3^{\epsilon}(x,t) + \tilde{R}_2^{\epsilon}(x,t).\\
   \text{where,}~&\tilde{R}_2^{\epsilon}(x,t)=\frac{x}{2 \sqrt{\pi \epsilon}}\frac{\epsilon^{\frac{3}{2}}}{\pi}\int_0^t \frac{e^{-\frac{1}{2 \epsilon}\Big[ \frac{x^2}{2(t-\tau)}-\tau\Big]}}{(t-\tau)^{\frac{3}{2}}}\int_0^\tau\frac{1-e^{-\frac{u}{2\epsilon}}}{u^{\frac{3}{2}}} \times\\&\Bigg[ \frac{1}{2(\epsilon (\tau-u))^{\frac{3}{2}}}\int_0^{\infty}\Big[e^{-\frac{(\tau-u)}{2\epsilon}}\theta^{\epsilon}_0(\xi)+\theta^{\epsilon}_0(-\xi)\Big]\xi e^{-\frac{\xi^2}{4\epsilon (\tau-u)}}d\xi-\frac{g(0)}{\sqrt{\epsilon (\tau-u)}}\Big(e^{-\frac{(\tau-u)}{2\epsilon}}+1\Big)\Bigg]\mathrm{d}u \mathrm{d}\tau\\
    &=\frac{x}{2 \sqrt{\pi \epsilon}}\frac{\epsilon^{\frac{3}{2}}}{\pi}\int_0^t \frac{e^{-\frac{1}{2 \epsilon}\Big[ \frac{x^2}{2(t-\tau)}-\tau\Big]}}{(t-\tau)^{\frac{3}{2}}} \int_0^\tau\frac{1}{2 \epsilon u^{\frac{1}{2}}} \int_{0}^1 e^{- \frac{\theta u}{2 \epsilon}} d \theta\times\\
      & \Bigg[ \frac{1}{2(\epsilon (\tau-u))^{\frac{3}{2}}}\int_0^{\infty}\Big[e^{-\frac{(\tau-u)}{2\epsilon}}\theta^{\epsilon}_0(\xi)+\theta^{\epsilon}_0(-\xi)\Big]\xi e^{-\frac{\xi^2}{4\epsilon (\tau-u)}}d\xi-\frac{g(0)}{\sqrt{\epsilon (\tau-u)}}\Big(e^{-\frac{(\tau-u)}{2\epsilon}}+1\Big)\Bigg]\mathrm{d}u \mathrm{d}\tau\\
      &= \frac{x}{(4 \pi \epsilon)^{\frac{3}{2}}} \int_0^t \int_0^\tau \int_0^1 \int_{0}^{\infty}
      \frac{\xi}{u^{\frac{1}{2}}(\tau-u)^{\frac{3}{2}}(t-\tau)^{\frac{3}{2}}}
      e^{-\frac{1}{2 \epsilon}\Big[ \frac{x^2}{2(t-\tau)}+\frac{\xi ^2}{2(\tau-u)}+(\theta-1) u+\int_{0}^\xi u_0 (y) dy \Big]} d \xi d\theta du d\tau\\
      & +\frac{x}{(4 \pi \epsilon)^{\frac{3}{2}}} \int_0^t \int_0^\tau \int_0^1 \int_{0}^{\infty}
      \frac{\xi}{ u^{\frac{1}{2}} (\tau-u)^{\frac{3}{2}}(t-\tau)^{\frac{3}{2}}}
      e^{-\frac{1}{2 \epsilon}\Big[ \frac{x^2}{2(t-\tau)}+\frac{\xi ^2}{2(\tau-u)}+\theta u-\tau +\int_{0}^{-\xi} u_0 (y) dy \Big]} d \xi d\theta du d\tau\\
     & -\frac{x g(0)}{4 \pi^{\frac{3}{2}} \epsilon^{\frac{1}{2}}}  \int_0^t \int_0^\tau \int_0^1  \frac{1}{ u^{\frac{1}{2}} (\tau-u)^{\frac{1}{2}}(t-\tau)^{\frac{3}{2}}}e^{-\frac{1}{2 \epsilon}\Big[ \frac{x^2}{2(t-\tau)}+(\theta-1) u\Big]} d\theta du d\tau\\
     &-\frac{x g(0)}{4 \pi^{\frac{3}{2}} \epsilon^{\frac{1}{2}}}  \int_0^t \int_0^\tau \int_0^1  \frac{1}{ u^{\frac{1}{2}} (\tau-u)^{\frac{1}{2}}(t-\tau)^{\frac{3}{2}}}e^{-\frac{1}{2 \epsilon}\Big[ \frac{x^2}{2(t-\tau)}-\tau +\theta u\Big]} d\theta du d\tau\\
     &=I_4^{\epsilon}(x,t)+I_5^{\epsilon}(x,t)-I_6^{\epsilon}(x,t)-I_7^{\epsilon}(x,t).
    \end{aligned}
\end{equation}
Therefore, combining  \eqref{expression for R}, \eqref{expression for R1},\eqref{expression for R2}, we get \eqref{definition 9}. Hence, the lemma follows.
\end{proof}
Similarly, we can prove the following lemma for $L^{\epsilon}(x,t).$
   \begin{lemma}\label{rewriting L}
   For all $(x,t)\in (-\infty,0)\times (0, \infty),$ $L^{\epsilon}(x,t)$ can be split as follows:
   \begin{equation}\label{expression for L}
   \begin{aligned}
     &L^{\epsilon}(x,t)=-J_1^{\epsilon} (x,t) +J_2^{\epsilon}(x,t)+J_3^{\epsilon}(x,t) +J_4^{\epsilon}(x,t)+J_5^{\epsilon}(x,t)-J_6^{\epsilon}(x,t)-J_7^{\epsilon}(x,t),\\
\text{where},~~
   &J_1^{\epsilon}(x,t)= \frac{1}{2\sqrt{\pi \epsilon t}} \int_0^\infty e^{-\frac{1}{2 \epsilon}[\frac{(x-\xi)^2}{2t}+\int_{0}^{-\xi}u_0 (y) dy]} d \xi\\
   &J_2^{\epsilon} (x,t)= \frac{1}{2\sqrt{\pi \epsilon t}} \int_0^\infty e^{-\frac{1}{2 \epsilon}[\frac{(x+\xi)^2}{2t}+\int_{0}^{-\xi} u_0 (y) dy]} d \xi\\
   &J_3^{\epsilon}(x,t)= -g(0) \frac{x}{2 \sqrt{\pi \epsilon}}\int_0^t I_0(-\frac{\tau}{4\epsilon})\frac{e^{-\frac{1}{4 \epsilon}\Big[ \frac{x^2}{t-\tau}+\tau\Big]}}{(t-\tau)^{\frac{3}{2}}} \, \mathrm{d}\tau\\
   &J_4^{\epsilon}(x,t)=  -\frac{x}{(4 \pi \epsilon)^{\frac{3}{2}}} \int_0^t \int_0^\tau \int_0^1 \int_{0}^{\infty}
      \frac{\xi e^{-\frac{1}{2 \epsilon}\Big[ \frac{x^2}{2(t-\tau)}+\frac{\xi ^2}{2(\tau-u)}+(\theta-1) u+\tau+\int_{0}^\xi u_0 (y) dy \Big]}}{u^{\frac{1}{2}}(\tau-u)^{\frac{3}{2}}(t-\tau)^{\frac{3}{2}}}
      d \xi d\theta du d\tau\\
   &J_5^{\epsilon}(x,t)=-\frac{x}{(4 \pi \epsilon)^{\frac{3}{2}}} \int_0^t \int_0^\tau \int_0^1 \int_{0}^{\infty}
      \frac{\xi e^{-\frac{1}{2 \epsilon}\Big[ \frac{x^2}{2(t-\tau)}+\frac{\xi ^2}{2(\tau-u)}+\theta u +\int_{0}^{-\xi} u_0 (y) dy \Big]}}{ u^{\frac{1}{2}} (\tau-u)^{\frac{3}{2}}(t-\tau)^{\frac{3}{2}}}
      d \xi d\theta du d\tau\\
   &J_6^{\epsilon}(x,t)= -\frac{x g(0)}{4 \pi^{\frac{3}{2}} \epsilon^{\frac{1}{2}}}  \int_0^t \int_0^\tau \int_0^1  \frac{1}{ u^{\frac{1}{2}} (\tau-u)^{\frac{1}{2}}(t-\tau)^{\frac{3}{2}}}e^{-\frac{1}{2 \epsilon}\Big[ \frac{x^2}{2(t-\tau)}+(\theta-1) u+\tau\Big]} d\theta du d\tau\\
     &J_7^{\epsilon}(x,t)=-\frac{x g(0)}{4 \pi^{\frac{3}{2}} \epsilon^{\frac{1}{2}}}  \int_0^t \int_0^\tau \int_0^1  \frac{1}{ u^{\frac{1}{2}} (\tau-u)^{\frac{1}{2}}(t-\tau)^{\frac{3}{2}}}e^{-\frac{1}{2 \epsilon}\Big[ \frac{x^2}{2(t-\tau)} +\theta u\Big]} d\theta du d\tau.
   \end{aligned}
   \end{equation}
\end{lemma}

\begin{lemma}\label{nonnegativity}
    Let $I_i^{\epsilon}(x,t), J_i^{\epsilon}(x,t), i=1,2,\cdots,7$, be defined as in the lemma \eqref{rewriting R} and \eqref{rewriting L}. Then \begin{equation*}
        \begin{aligned}
            &I_1^{\epsilon}(x,t)-I_2^{\epsilon}(x,t)\geq 0,\\&
            I_3^{\epsilon}(x,t)-I_6^{\epsilon}(x,t)-I_7^{\epsilon}(x,t)\geq 0, ~~~\forall x,t>0
        \end{aligned}
    \end{equation*}
    and \begin{equation*}
        \begin{aligned}
            &-J_1^{\epsilon}(x,t)+J_2^{\epsilon}(x,t)\geq 0,\\&
            J_3^{\epsilon}(x,t)-J_6^{\epsilon}(x,t)-J_7^{\epsilon}(x,t)\geq 0, ~~~\forall x<0,t>0.
        \end{aligned}
    \end{equation*}
\end{lemma}
\begin{proof}
    {\bf Step-1:}For all $x,\xi>0$, $(x-\xi)^2\leq (x+\xi)^2.$ So we have the following.
    \begin{equation*}
         e^{-\frac{1}{2 \epsilon}[\frac{(x-\xi)^2}{2t}+\int_{0}^\xi u_0 (y) dy]} \geq e^{-\frac{1}{2 \epsilon}[\frac{(x+\xi)^2}{2t}+\int_{0}^\xi u_0 (y) dy]}.
    \end{equation*} This immediately follows that $I_1^{\epsilon}(x,t)-I_2^{\epsilon}(x,t)\geq 0, ~~~\forall x,t>0.$

    \noindent
    {\bf Step-2:} From \eqref{expression for R2}, we can write 
    \begin{equation}\label{definition 11}
        \begin{aligned}
             &I_3^{\epsilon}(x,t)-I_6^{\epsilon}(x,t)-I_7^{\epsilon}(x,t)\\=
             &g(0)\Bigg[ \frac{x}{2 \sqrt{\pi \epsilon}}\int_0^t I_0(-\frac{\tau}{4\epsilon})\frac{e^{-\frac{1}{4 \epsilon}\Big[ \frac{x^2}{t-\tau}-\tau\Big]}}{(t-\tau)^{\frac{3}{2}}} \, \mathrm{d}\tau\\&-\frac{x}{4 \pi^{\frac{3}{2}} \epsilon^{\frac{1}{2}}}  \int_0^t \int_0^\tau \int_0^1  \frac{1}{ u^{\frac{1}{2}} (\tau-u)^{\frac{1}{2}}(t-\tau)^{\frac{3}{2}}}e^{-\frac{1}{2 \epsilon}\Big[ \frac{x^2}{2(t-\tau)}+(\theta-1) u\Big]} d\theta du d\tau\\&-\frac{x}{4 \pi^{\frac{3}{2}} \epsilon^{\frac{1}{2}}}  \int_0^t \int_0^\tau \int_0^1  \frac{1}{ u^{\frac{1}{2}} (\tau-u)^{\frac{1}{2}}(t-\tau)^{\frac{3}{2}}}e^{-\frac{1}{2 \epsilon}\Big[ \frac{x^2}{2(t-\tau)}-\tau +\theta u\Big]} d\theta du d\tau\Bigg]\\=
             &g(0)F^{\epsilon}(x,t).
        \end{aligned}
    \end{equation}
    So from \eqref{expression for R} we have the coefficient of $g(0)$ in $R^{\epsilon}(x,t)$ is $F^{\epsilon}(x,t)$.
    Now from \eqref{transformation 1} and \eqref{definition 3}, we have the sign of the coefficient of $g(0)$ in $R^{\epsilon}(x,t)$ depends only on the sign of the coefficient of $g(0)$ in $g$.
 Now from \eqref{definition 1} and \eqref{definition 2} we get \begin{equation*}
 \begin{aligned}
     g(t)=&\frac{\epsilon^{\frac{3}{2}}}{\pi}\int_0^t\frac{1-e^{-\frac{\tau}{2\epsilon}}}{\tau^{\frac{3}{2}}}\Big[\frac{1}{2(\epsilon (t-\tau))^{\frac{3}{2}}}\int_0^{\infty}\Big[e^{-\frac{t-\tau}{2\epsilon}}\theta^{\epsilon}_0(\xi)+\theta^{\epsilon}_0(-\xi)\Big]\xi e^{-\frac{\xi^2}{4\epsilon (t-\tau)}}d\xi\Big]d\tau\\&+g(0)\Bigg[e^{-\frac{t}{4\epsilon}}I_0(-\frac{t}{4\epsilon})-\frac{\epsilon^{\frac{3}{2}}}{\pi}\int_0^t\frac{1-e^{-\frac{\tau}{2\epsilon}}}{\tau^{\frac{3}{2}}}\Big(\frac{e^{-\frac{t-\tau}{2\epsilon}}+1}{\sqrt{\epsilon(t-\tau)}}\Big)d\tau\Bigg].
 \end{aligned}
 \end{equation*}
Let's denote the coefficient of $g(0)$ in $g$ by $G^{\epsilon}(t)$. Now we claim that $G^{\epsilon}(t)$ is non-negative. Suppose the contrary. Now choose $u_0(x)=\begin{cases}
    M, & x>0\\
    -M, & x<0,
\end{cases}$  where $M$ is very large number. Then 
\begin{equation*}
    \begin{aligned}
        &\theta_0^{\epsilon}(\xi)=e^{-\frac{M\xi}{2\epsilon}}\\
        &\theta_0^{\epsilon}(-\xi)=e^{-\frac{M\xi}{2\epsilon}},
    \end{aligned}
\end{equation*}
which imply the first term in $g(t)$ is very small. Hence $g(t)$ becomes negative for some $t\geq0$, which contradicts the result that $g(t)$ is positive for all $t\geq0$ (See \cite{chung}, page-2530, proof of proposition 2.4). Hence the claim.
Therefore we have $F^{\epsilon}(x,t)$ is non-negative for all $t\geq0$. Then \eqref{definition 11} implies that $$I_3^{\epsilon}(x,t)-I_6^{\epsilon}(x,t)-I_7^{\epsilon}(x,t)\geq 0, ~~~\forall x,t>0.$$Similarly we can prove for $J_i^{\epsilon}(x,t)$. Hence the lemma follows.
\end{proof}
\subsection*{\textbf{Proof of theorem \ref{vanishing viscosity limit}:}}
\begin{proof}
{\bf Step-1:}Let's denote $U(x,t)=\min(\Big[\frac{x^2}{2(t-\tau)}-\tau\Big], U_R^ 1(x,t), U_R^ 2(x,t), U_R^3(x,t))$.  Observe that 

\noindent
\begin{equation*}
\begin{aligned}
  I_1^{\epsilon}(x,t) &\leq e^{-\frac{U_R^3(x,t)}{2\epsilon}} e^{\frac{U_R^3(x,t)}{2}}\frac{I_1 ^{1}(x,t)}{\sqrt{\epsilon}}\\  
I_2^{\epsilon}(x,t) &\leq e^{-\frac{U_R^3(x,t)}{2\epsilon}} e^{\frac{1}{2}\displaystyle{\min_{\xi \geq 0}}\Big[ \frac{(x+\xi)^2}{2 t} + \int_{0}^{\xi} u_0(z)dz\Big]}\frac{I_2 ^{1}(x,t)}{\sqrt{\epsilon}}\\
I_3^{\epsilon}(x,t)&=g(0) \frac{x}{2 \sqrt{\pi \epsilon}}\int_0^t I_0(-\frac{\tau}{4\epsilon})\frac{e^{-\frac{1}{4 \epsilon}\Big[ \frac{x^2}{t-\tau}-\tau\Big]}}{(t-\tau)^{\frac{3}{2}}} \, \mathrm{d}\tau\\
 &= g(0) \frac{x}{2 \sqrt{\pi \epsilon}}\int_0^t \int_0 ^1 \frac{e^{-\frac{1}{2 \epsilon}\Big[ \frac{x^2}{2(t-\tau)}-\tau \sin^2(\frac{\pi \theta}{2})\Big]}}{(t-\tau)^{\frac{3}{2}}} \, \mathrm{d}\tau\\
 &\leq e^{-\frac{1}{2\epsilon} \displaystyle{\min_{0\leq \tau<t}}\Big[\frac{x^2}{2(t-\tau)}-\tau\Big]  } e^{\frac{1}{2} \displaystyle{\min_{0\leq \tau<t}}\Big[\frac{x^2}{2(t-\tau)}-\tau\Big]} \frac{I_3 ^{1}(x,t)}{\sqrt{\epsilon}}\\
 I_4^{\epsilon}(x,t) &\leq e^{-\frac{U_R^1(x,t)}{2\epsilon}} e^{\frac{U_R^1(x,t)}{2}}\frac{I_4 ^{1}(x,t)}{\epsilon^{\frac{3}{2}}}\\
I_5^{\epsilon}(x,t) &\leq e^{-\frac{U_R^2(x,t)}{2\epsilon}} e^{\frac{U_R^2(x,t)}{2}}\frac{I_5 ^{1}(x,t)}{\epsilon^{\frac{3}{2}}}\\
I_6 ^{\epsilon}(x,t) &\leq e^{-\frac{1}{2\epsilon} \displaystyle{\min_{0\leq \tau<t}}\Big[\frac{x^2}{2(t-\tau)}-\tau\Big]  } e^{\frac{1}{2} \displaystyle{\min_{0\leq \tau<t}}\Big[\frac{x^2}{2(t-\tau)}-\tau\Big]} \frac{I_6 ^{1}(x,t)}{\sqrt{\epsilon}}\\
I_7 ^{\epsilon}(x,t) &\leq e^{-\frac{1}{2\epsilon} \displaystyle{\min_{0\leq \tau<t}}\Big[\frac{x^2}{2(t-\tau)}-\tau\Big]  } e^{\frac{1}{2} \displaystyle{\min_{0\leq \tau<t}}\Big[\frac{x^2}{2(t-\tau)}-\tau\Big]} \frac{I_7 ^{1}(x,t)}{\sqrt{\epsilon}}.\\
 \text{Then,}~ &2\epsilon\log (R^\epsilon(x,t))\\&\leq 2\epsilon \log \Big[e^{-\frac{t}{2\epsilon}}\Big(I_1^{\epsilon}(x,t) +I_2^{\epsilon}(x,t) +I_3^{\epsilon}(x,t)+I_4^{\epsilon}(x,t)+I_5^{\epsilon}(x,t)+I_6^{\epsilon}(x,t)+I_7^{\epsilon}(x,t)\Big)\Big]\\
  &\leq 2\epsilon \log \Big[e^{-\frac{U(x,t)+t}{2\epsilon}}\Big(e^{-\frac{U_R^3(x,t)-U(x,t)}{2\epsilon}} e^{\frac{U_R^3(x,t)}{2}}\frac{I_1 ^{1}(x,t) }{\sqrt{\epsilon}} \\&+e^{-\frac{U_R^3(x,t)-U(x,t)}{2\epsilon}} e^{\frac{1}{2}\displaystyle{\min_{\xi \geq 0}}\Big[ \frac{(x+\xi)^2}{2 t} + \int_{0}^{\xi} u_0(z)dz\Big]}\frac{I_2 ^{1}(x,t) }{\sqrt{\epsilon}}\\&+e^{-\frac{1}{2\epsilon} \Big(\displaystyle{\min_{0\leq \tau<t}}\Big[\frac{x^2}{2(t-\tau)}-\tau\Big]-U(x,t)\Big) } e^{\frac{1}{2} \displaystyle{\min_{0\leq \tau<t}}\Big[\frac{x^2}{2(t-\tau)}-\tau\Big]} \frac{I_3 ^{1}(x,t) }{\sqrt{\epsilon}}\\ &+ e^{-\frac{U_R^1(x,t)-U(x,t)}{2\epsilon}} e^{\frac{U_R^1(x,t)}{2}}\frac{I_4 ^{1}(x,t) }{\epsilon^{\frac{3}{2}}}+e^{-\frac{U_R^2(x,t)-U(x,t)}{2\epsilon}} e^{\frac{U_R^2(x,t)}{2}}\frac{I_5 ^{1}(x,t) }{\epsilon^{\frac{3}{2}}}\\&+e^{-\frac{1}{2\epsilon} \Big(\displaystyle{\min_{0\leq \tau<t}}\Big[\frac{x^2}{2(t-\tau)}-\tau\Big]-U(x,t)\Big) } e^{\frac{1}{2} \displaystyle{\min_{0\leq \tau<t}}\Big[\frac{x^2}{2(t-\tau)}-\tau\Big]} \frac{I_6 ^{1}(x,t) }{\sqrt{\epsilon}}\\
\end{aligned}    
\end{equation*}
\begin{equation*}
    \begin{aligned}
        &+ e^{-\frac{1}{2\epsilon} \Big(\displaystyle{\min_{0\leq \tau<t}}\Big[\frac{x^2}{2(t-\tau)}-\tau\Big]-U(x,t)\Big) } e^{\frac{1}{2} \displaystyle{\min_{0\leq \tau<t}}\Big[\frac{x^2}{2(t-\tau)}-\tau\Big]} \frac{I_7 ^{1}(x,t) }{\sqrt{\epsilon}}\Big)\Big].
    \end{aligned}
\end{equation*}
This implies \begin{equation}\label{limsup}
    \displaystyle{\limsup_{\epsilon \to 0}}[ 2\epsilon\log (R^\epsilon(x,t))] \leq -\Big(U(x,t)+t\Big).
\end{equation}
{\bf Step-2:} 
\begin{equation*}
\begin{aligned}
U(x,t)&=\min(\Big[\frac{x^2}{2(t-\tau)}-\tau\Big], U_R^ 1(x,t), U_R^ 2(x,t), U_R^3(x,t))\\
      &=\min(U_R^ 1(x,t), U_R^ 2(x,t), U_R^3(x,t)).
\end{aligned}   
\end{equation*}
We derive the lower limit considering the following two cases.

\noindent
{\it Case1:} $U(x,t)= \min(U_R^ 1(x,t), U_R^ 2(x,t)).$ 
As by lemma \eqref{nonnegativity}, $$I_1^{\epsilon}(x,t)-I_2^{\epsilon}(x,t), I_3^{\epsilon}(x,t)-I_6^{\epsilon}(x,t)-I_7^{\epsilon}(x,t)\geq 0,~~~~~\forall x,t>0,$$ from \eqref{definition 9} we have 
\begin{equation}\label{definition 12}
\begin{aligned}
 R^{\epsilon}(x,t)\geq e^{-\frac{t}{2 \epsilon}} [I_4^{\epsilon}(x,t)+I_5^{\epsilon}(x,t)].\\
 \end{aligned}
\end{equation}
Now we find lower bounds for $I_4^{\epsilon}(x,t),I_5^{\epsilon}(x,t).$

\noindent
Suppose $\Big[\frac{x^2}{2(t-\tau)}+\frac{\xi ^2}{2(\tau-u)}+(\theta-1) u+\int_{0}^\xi u_0 (y) dy \Big]$ achieves minimum at a unique point $\tau=\tau_1, \xi=\xi_1, u=u_1,\theta=0$ (proved in theorem \eqref{uniqueness}). For given $\delta_2>0$, there exists a $\delta_1>0$ such that $$\Big| \frac{x^2}{2(t-\tau)}+\frac{\xi ^2}{2(\tau-u)}+(\theta-1) u+\int_{0}^\xi u_0 (y) dy-U_R^ 1(x,t) \Big| \leq \delta_2$$ for $|\tau-\tau_1|, |\xi-\xi_1|, |u-u_1|,|\theta|\leq \delta_1.$ Then,
\begin{equation}\label{I4}
\begin{aligned}
  &I_4^{\epsilon}(x,t)\\&= \frac{x}{(4 \pi \epsilon)^{\frac{3}{2}}} \int_0^t \int_0^\tau \int_0^1 \int_{0}^{\infty}
      \frac{\xi e^{-\frac{1}{2 \epsilon}\Big[ \frac{x^2}{2(t-\tau)}+\frac{\xi ^2}{2(\tau-u)}+(\theta-1) u+\int_{0}^\xi u_0 (y) dy \Big]} }{u^{\frac{1}{2}}(\tau-u)^{\frac{3}{2}}(t-\tau)^{\frac{3}{2}}}
      d \xi d\theta du d\tau\\  
      &\geq \frac{x}{(4 \pi \epsilon)^{\frac{3}{2}}} \int_{\tau_1}^{\tau_1+\delta_1} \int_{u_1}^{u_1+\delta_1} \int_{0}^{\delta_1} \int_{\xi_1}^{\xi+\delta_1}
      \frac{\xi  e^{-\frac{1}{2 \epsilon}\Big[ \frac{x^2}{2(t-\tau)}+\frac{\xi ^2}{2(\tau-u)}+(\theta-1) u+\int_{0}^\xi u_0 (y) dy \Big]}}{u^{\frac{1}{2}}(\tau-u)^{\frac{3}{2}}(t-\tau)^{\frac{3}{2}}} d \xi d\theta du d\tau\\
      &\geq \frac{x}{(4 \pi \epsilon)^{\frac{3}{2}}}e^{-\frac{U_R^1(x,t)+\delta_2}{2\epsilon}} \int_{\tau_1}^{\tau_1+\delta_1} \int_{u_1}^{u_1+\delta_1} \int_{0}^{\delta_1} \int_{\xi_1}^{\xi_1+\delta_1}
      \frac{\xi }{u^{\frac{1}{2}}(\tau-u)^{\frac{3}{2}}(t-\tau)^{\frac{3}{2}}} d \xi d\theta du d\tau\\
      &=\frac{x}{(4 \pi \epsilon)^{\frac{3}{2}}}e^{-\frac{U_R^1(x,t)+\delta_2}{2\epsilon}}A(\delta_1).
\end{aligned}
\end{equation}
Similarly, suppose $\Big[ \frac{x^2}{2(t-\tau)}+\frac{\xi ^2}{2(\tau-u)}+\theta u-\tau +\int_{0}^{-\xi} u_0 (y) dy \Big]$ achieves minimum at a unique point $\tau=\tau_2, \xi=\xi_2,\theta=0, u=0$  (proved in theorem \eqref{uniqueness}).Then for $\delta_2 >0,$ there exists a $\delta_3>0$ such that $$\Big| \frac{x^2}{2(t-\tau)}+\frac{\xi ^2}{2(\tau-u)}+\theta u-\tau +\int_{0}^{-\xi} u_0 (y) dy-U_R^ 2(x,t)\Big|\leq \delta_2 $$ for $|\tau-\tau_2|, |\xi-\xi_2|,|\theta|, |u|\leq \delta_3.$ Then,
\begin{equation}\label{I5}
    \begin{aligned}
        I_5^{\epsilon}(x,t)&\geq \frac{x}{(4 \pi \epsilon)^{\frac{3}{2}}}e^{-\frac{U_R^2(x,t)+\delta_2}{2\epsilon}} \int_{\tau_2}^{\tau_2+\delta_3} \int_{0}^{\delta_3} \int_{0}^{\delta_3} \int_{\xi_2}^{\xi_2+\delta_3}
      \frac{\xi }{u^{\frac{1}{2}}(\tau-u)^{\frac{3}{2}}(t-\tau)^{\frac{3}{2}}} d \xi d\theta du d\tau\\&=\frac{x}{(4 \pi \epsilon)^{\frac{3}{2}}}e^{-\frac{U_R^2(x,t)+\delta_2}{2\epsilon}}B(\delta_3).
    \end{aligned}
\end{equation}
Therefore, combining \eqref{definition 12}, \eqref{I4} and \eqref{I5}, we get
      \begin{equation*}
          \begin{aligned}
              &2\epsilon\log (R^\epsilon(x,t))\\
              &\geq 2\epsilon \log\Big(e^{-\frac{t}{2 \epsilon}} \Big[I_4^{\epsilon}(x,t)+I_5^{\epsilon}(x,t)\Big]\Big)\\
              &\geq 2\epsilon \log\Big[\frac{x}{(4 \pi \epsilon)^{\frac{3}{2}}}e^{-\frac{t+\delta_2}{2 \epsilon}}\Big(e^{-\frac{U_R^1(x,t)}{2\epsilon}}A(\delta_1)+e^{-\frac{U_R^2(x,t)}{2\epsilon}}B(\delta_3)\Big)\Big]\\
              &\geq 2\epsilon \log\Big[\frac{x}{(4 \pi \epsilon)^{\frac{3}{2}}}e^{-\frac{t+\delta_2+U(x,t)}{2 \epsilon}}\Big(e^{-\frac{U_R^1(x,t)-U(x,t)}{2\epsilon}}A(\delta_1)+e^{-\frac{U_R^2(x,t)-U(x,t)}{2\epsilon}}B(\delta_3)\Big)\Big].
          \end{aligned}
      \end{equation*} Therefore,
          $\displaystyle{\liminf_{\epsilon \to 0}} \Big[2\epsilon\log (R^\epsilon)(x,t)\Big] \geq -\Big[t+\delta_2+\min\{U_R^1(x,t),U_R^2(x,t)\}\Big].$
        Since $\delta_2$ is arbitrary, we have \begin{equation}\label{liminf 1}
          \displaystyle{\liminf_{\epsilon \to 0}} \Big[2\epsilon\log (R^\epsilon)(x,t)\Big] \geq -\Big[t+U(x,t)\Big].
      \end{equation}
      
      \noindent
     {\it Case 2:} $U_R^ 3(x,t) < \min(U_R^ 1(x,t), U_R^ 2(x,t)).$ 
     
     \noindent
In this case, \begin{equation}\label{definition 13}
    R^{\epsilon}(x,t)\geq I_1^{\epsilon}(x,t)-I_2^{\epsilon}(x,t)
\end{equation}
 as $I_4^{\epsilon}(x,t),I_5^{\epsilon}(x,t)\geq 0$ and from lemma \eqref{nonnegativity} $I_3^{\epsilon}(x,t)-I_6^{\epsilon}(x,t)-I_7^{\epsilon}(x,t)\geq 0.$
Now Observe that
\begin{equation}\label{I2}
  \begin{aligned}
      I_2 ^{\epsilon}(x,t) &\leq  e^{-\frac{1}{2\epsilon}\displaystyle{\min_{\xi \geq 0}}\Big[ \frac{(x+\xi)^2}{2 t} + \int_{0}^{\xi} u_0(z)dz\Big]} e^{\frac{1}{2}\displaystyle{\min_{\xi \geq 0}}\Big[ \frac{(x+\xi)^2}{2 t} + \int_{0}^{\xi} u_0(z)dz\Big]}\frac{I_2 ^{1}(x,t)}{\sqrt{\epsilon}}\\
      &=e^{-\frac{1}{2\epsilon}\displaystyle{\min_{\xi \geq 0}}\Big[ \frac{(x+\xi)^2}{2 t} + \int_{0}^{\xi} u_0(z)dz\Big]} B_1(x,t)\frac{I_2 ^{1}(x,t)}{\sqrt{\epsilon}}.
  \end{aligned}  
\end{equation}
Suppose $U_R^ 3(x,t)$ attains minimum at $\xi=\xi_3,$ then given $\delta_5 >0,$ there exists a $\delta_4 >0$ such that $|\frac{(x-\xi)^2}{2t}+\int_{0}^\xi u_0 (y) dy-U_R^ 3(x,t)|\leq \delta_5$ for $\xi_3\leq \xi\leq \xi_3+\delta_4.$
We have the following estimation.
\begin{equation}\label{I1}
    \begin{aligned}
        I_1^{\epsilon}(x,t)&= \frac{1}{2\sqrt{\pi \epsilon t}} \int_0^\infty e^{-\frac{1}{2 \epsilon}\Big[\frac{(x-\xi)^2}{2t}+\int_{0}^\xi u_0 (y) dy\Big]} d \xi\\
       & \geq \frac{1}{2\sqrt{\pi \epsilon t}} \int_{\xi_3}^ {\xi_3+\delta_4} e^{-\frac{1}{2 \epsilon}\Big[\frac{(x-\xi)^2}{2t}+\int_{0}^\xi u_0 (y) dy\Big]} d \xi\\
       &\geq \frac{e^{-\frac{U_R^3(x,t)}{2\epsilon}}}{2\sqrt{\pi \epsilon t}} \int_{\xi_3}^ {\xi_3+\delta_4}e^{-\frac{1}{2 \epsilon}\Big[\frac{(x-\xi)^2}{2t}+\int_{0}^\xi u_0 (y) dy-U_R^3(x,t)\Big]}d\xi\\
       & \geq \frac{\delta_4}{2\sqrt{\pi \epsilon t}} e^{-\frac{U_R^3(x,t)+\delta_5}{2\epsilon}}.
    \end{aligned}
\end{equation}
Therefore, combining \eqref{definition 13}, \eqref{I2} and \eqref{I1}, we get the following.
      \begin{equation*}
          \begin{aligned}
              &2\epsilon\log (R^\epsilon(x,t))\\
              \geq& 2\epsilon \log\Big(e^{-\frac{t}{2 \epsilon}} [I_1^\epsilon(x,t)-I_2^\epsilon(x,t)]\Big)\\
              \geq& 2\epsilon \log\Big[e^{-\frac{t}{2\epsilon}}\Big(\frac{\delta_4}{2\sqrt{\pi \epsilon t}} e^{-\frac{U_R^3(x,t)+\delta_5}{2\epsilon}}-e^{-\frac{1}{2\epsilon}\displaystyle{\min_{\xi \geq 0}}\Big[ \frac{(x+\xi)^2}{2 t} + \int_{0}^{\xi} u_0(z)dz\Big]} B_1(x,t) \frac{I_2 ^{1}(x,t)}{\sqrt{\epsilon}}\Big)\Big]\\
              \geq& 2\epsilon \log\Bigg[ e^{-\frac{t+U_R^3(x,t)+\delta_5}{2\epsilon}}\Big(\frac{\delta_4}{2\sqrt{\pi \epsilon t}}\\&-e^{-\frac{1}{2\epsilon}\displaystyle{\min_{\xi \geq 0}}\Big[ \frac{(x+\xi)^2}{2 t} + \int_{0}^{\xi} u_0(z)dz-(U_R^3(x,t)+\delta_5)\Big]}B_1(x,t) \frac{I_2 ^{1}(x,t)}{\sqrt{\epsilon}}\Big)\Bigg]
          \end{aligned}
      \end{equation*}
  Therefore, $\displaystyle{\liminf_{\epsilon \to 0}} \Big[2\epsilon\log (R^\epsilon(x,t)) \Big]\geq -\Big[t+\delta_5+U_R^3(x,t)\Big].$
Since here $U(x,t)=U_R^3(x,t)$ and $\delta_5$ is arbitrary, we have \begin{equation}\label{liminf 2}
          \displaystyle{\liminf_{\epsilon \to 0}} \Big[2\epsilon\log (R^\epsilon)(x,t)\Big] \geq -\Big[t+U(x,t)\Big].
      \end{equation}
      Therefore  for both cases  we get$$\displaystyle{\lim_{\epsilon \to 0}} \Big[2\epsilon\log (R^\epsilon)(x,t)\Big] = -\Big[t+U(x,t)\Big]$$ from \eqref{limsup}, \eqref{liminf 1} and \eqref{liminf 2}.
      Hence $$\displaystyle{\lim_{\epsilon \to 0}} \Big[-2\epsilon\log (R^\epsilon)(x,t)\Big] = U(x,t)+t=U_R(x,t)$$
      Similarly, we can prove that the approximations $[-2\epsilon \log L^{\epsilon}(x,t)]$ converge to $U_L(x,t).$Hence the theorem.
\end{proof}

\section{Solution of the inviscid equation \ref{inviscid equation}}
In this section, we obtain an explicit formula for the weak limit $u(x,t)$ of  $u^{\epsilon}(x,t).$ This limit is the weak derivative of $U,$ where $U$ is the pointwise limit of $-2 \epsilon \log(\theta^{\epsilon})$ obtained in the last section. This $u(x,t)$ is indeed a solution of the inviscid equation \eqref{inviscid equation} (see \cite{MR2028700}). To obtain the weak limit of $u^\epsilon,$ we need the following lemmas. The first lemma shows the non-intersecting property of each of the functionals $U_L^1, U_L^2, U_L ^3, U_R^1, U_R^2$ and $U_R^3$. For convenience, we represent functionals $U_L^1, U_L^2, U_R^1, U_R^2$ using one notation: 
\begin{equation}
\label{notation}
        \begin{aligned}
            U^i_I (x, t, \tau, \xi,u)
    = \begin{cases}
        \frac{x^2}{2(t-\tau)}+\frac{\xi^2}{2 (\tau+(i-2)u)}+(2-i)(\tau-u)\\+\int_{0}^{(3-2i)\xi} u_0(z)dz,& x<0,\xi>0,I=L\\
        \frac{x^2}{2(t-\tau)}+\frac{\xi^2}{2 (\tau+(i-2)u)}+(1-i)\tau+(i-2)u \\+\int_{0}^{(3-2i)\xi} u_0(z)dz,&x>0,\xi>0,I=R
    \end{cases}
        \end{aligned}
        \end{equation} and 
\begin{equation*}
\begin{aligned}
     &U_L^3(x,t,\xi)= \frac{(x+\xi)^2}{2 t} + \int_{0}^{-\xi} u_0(z)dz,~~x<0,\xi>0,\\
     & U_R^3(x,t,\xi)= \frac{(x-\xi)^2}{2 t} + \int_{0}^{\xi} u_0(z)dz,~~x>0,\xi>0.
\end{aligned}
\end{equation*}
\begin{lemma}\label{nonintersecting property}
For $i\in\{1,2\}, I\in \{L, R\}$  the minimizers of the functionals $U^i_I (x, t, \tau, \xi)$ and  $U^3_I (x, t, \xi)$ satisfies the non intersecting property in the following sense.

\begin{enumerate}
   \item Suppose $\min[U^i_L (x, t, \tau, \xi):0 \le \tau < t, 0 \le u < \tau, \xi \ge 0]$
  is achieved at $\tau=\tau_L ^i,\xi=\xi_L ^i$ and for $x_1 < x,\min[U^i_L (x_1, t, \tau, \xi):0 \le \tau < t, 0 \le u < \tau, \xi \ge 0]$ is achieved at $\tau=\bar{\tau}_L ^i$ and $\xi=\bar{\xi}_L ^i$. Then $\bar{\tau}_L ^i\leq \tau_L ^i$ and $\bar{\xi}_L ^i \leq \xi_L ^i$.
   \item Suppose $\min[U^i_R (x, t, \tau, \xi):0 \le \tau < t, 0 \le u < \tau, \xi \ge 0]$ is
   achieved at $\tau=\tau_R ^i,\xi=\xi_R ^i$ and for $x_1 > x,\min[U^i_R (x_1, t, \tau, \xi):0 \le \tau < t, 0 \le u < \tau, \xi \ge 0]$is achieved at $\tau=\bar{\tau}_R ^i$ and $\xi=\bar{\xi}_R ^i$. Then $\bar{\tau}_R ^i\leq \tau_R ^i$ and $\bar{\xi}_R ^i \leq \xi_R ^i$.
    \item Suppose  $\min[U^3_L (x, t, \xi):\xi \geq 0]$ is
   achieved at $\xi=\xi_1$ and for $x_1 < x,\min[U^3_L (x_1, t, \xi):\xi \geq 0]$ is achieved at $\xi=\xi_2.$ Then $\xi_2 \leq \xi_1.$
   \item Suppose  $\min[U^3_R (x, t, \xi):\xi \geq 0]$ is
   achieved at $\xi=\xi_1$ and for $x_1 >x,\min[U^3_R (x_1, t, \xi):\xi \geq 0]$ is achieved at $\xi=\xi_2.$ Then $\xi_2 \leq \xi_1.$
          \end{enumerate}
\end{lemma}
\begin{proof}
    {\bf Step 1:} Let $(\bar{x}, \bar{t})$ be a point on the line joining $(x,t)$ to $(0,\tau_L ^i)$. Then we show that $\displaystyle \min_{0\leq \tau <\bar{t}, ~\xi\geq 0}U^i_L( \bar{x}, \bar{t}, \tau, \xi)$ will be realized for unique $\tau=\tau_L ^i$ and $\xi=\xi_L ^i$ for all $i$. This implies that the characteristics lines can not cross in left quarter plane. This in turn implies $\bar{\tau}_L ^i \leq \tau_L ^i$ for any $\bar{\tau}_L ^i<\bar{t}.$

Since $U^i _L (x, t, \tau, \xi)$ attains minimum for $\tau=\tau_L ^i$ and $\xi=\xi_L ^i,$ then for any other $\bar{\tau}_L ^i< \bar{t}$ and $\bar{\xi}_L ^i\geq 0$, we get  $$U^i _L (x, t, \tau_L ^i, \xi_L ^i)-U^i _L (x, t, \bar{\tau}_L ^i, \bar{\xi}_L ^i)\leq 0~~~~\forall i.$$
This implies  
\begin{equation*}
    \begin{aligned}
        &\frac{x^2}{2(t-\tau_L ^i)}+\frac{{\xi_L ^i}^2}{2 (\tau_L ^i+(i-2)u)}+(2-i)(\tau_L ^i-u)+\int_{0}^{(3-2i)\xi_L ^i} u_0(z)dz\\&-\frac{x^2}{2(t-\bar{\tau}_L ^i)}-\frac{\bar{{\xi}_L ^i}^2}{2 (\bar{\tau}_L ^i+(i-2)u)}-(2-i)(\bar{\tau}_L ^i-u)-\int_{0}^{(3-2i)\bar{\xi}_L ^i} u_0(z)dz\leq 0.
    \end{aligned}
\end{equation*}
We may rewrite the above as 
\begin{equation}\label{inequality}
    \frac{1}{2}\int_{\bar{\tau}_L ^i}^{\tau_L ^i} \frac{x^2}{(t-s)^2} ds+\frac{{\xi_L ^i}^2}{2 (\tau_L ^i+(i-2)u)}-\frac{\bar{{\xi}_L ^i}^2}{2 (\bar{\tau}_L ^i+(i-2)u)} -\int_{(3-2i)\xi_L ^i}^{(3-2i)\bar{\xi}_L ^i} u_0(z)dz +(2-i)(\tau_L ^i-\bar{\tau}_L ^i)\leq 0.
\end{equation}
Now, 
\begin{equation*}
    \begin{aligned}
    & U^i _L (\bar{x}, \bar{t}, \tau_L ^i, \xi_L ^i)-U^i _L (\bar{x}, \bar{t}, \bar{\tau}_L ^i, \bar{\xi}_L ^i)\\
   = & \frac{\bar{x}^2}{2(\bar{t}-\tau_L ^i)}+\frac{{\xi_L ^i}^2}{2 (\tau_L ^i+(i-2)u)}+(2-i)(\tau_L ^i-u)+\int_{0}^{(3-2i)\xi_L ^i} u_0(z)dz\\&-\frac{\bar{x}^2}{2(\bar{t}-\bar{\tau}_L ^i)}-\frac{\bar{{\xi}_L ^i}^2}{2 (\bar{\tau}_L ^i+(i-2)u)}-(2-i)(\bar{\tau}_L ^i-u)-\int_{0}^{(3-2i)\bar{\xi}_L ^i} u_0(z)dz\\
    =& \frac{1}{2} \int_{\bar{\tau}_L ^i}^{\tau_L ^i} \frac{\bar{x}^2}{(\bar{t}-s)^2} ds +\frac{{\xi_L ^i}^2}{2 (\tau_L ^i+(i-2)u)}-\frac{\bar{{\xi}_L ^i}^2}{2 (\bar{\tau}_L ^i+(i-2)u)} \\&-\int_{(3-2i)\xi_L ^i}^{(3-2i)\bar{\xi}_L ^i} u_0(z)dz+(2-i)(\tau_L ^i-\bar{\tau}_L ^i),~~\forall i.
    \end{aligned}
\end{equation*}
Observe 
\begin{equation*}
    \begin{aligned}
    \frac{t-s}{x}&=\frac{t-\tau_L ^i}{x} +\frac{\tau_L ^i-s}{x}\\
                &=\frac{\bar{t}-\tau_L ^i}{\bar{x}} +\frac{\tau_L ^i-s}{x}\\
                &=\frac{\bar{t}-\tau_L ^i}{\bar{x}}+\frac{\tau_L ^i-s}{\bar{x}}- \frac{\tau_L ^i-s}{\bar{x}}+\frac{\tau_L ^i-s}{x}\\
                &=\frac{\bar{t}-s}{\bar{x}}+ (\tau_L ^i-s)(\frac{1}{x}-\frac{1}{\bar{x}}).
    \end{aligned}
\end{equation*}
From the above calculation, we conclude that $\frac{x^2}{(t-s)^2}<\frac{\bar{x}^2}{(\bar{t}-s)^2}.$
Therefore, \begin{equation*}
    \begin{aligned}
    & U^i _L (\bar{x}, \bar{t},\tau_L ^i, \xi_L ^i)-U^i_L (\bar{x}, \bar{t}, \bar{\tau}_L ^i, \bar{\xi}_L ^i)\\
    =& \frac{1}{2} \int_{\bar{\tau}_L ^i}^{\tau_L ^i} \frac{\bar{x}^2}{(\bar{t}-s)^2} ds +\frac{{\xi_L ^i}^2}{2 (\tau_L ^i+(i-2)u)}-\frac{\bar{{\xi}_L ^i}^2}{2 (\bar{\tau}_L ^i+(i-2)u)} \\&-\int_{(3-2i)\xi_L ^i}^{(3-2i)\bar{\xi}_L ^i} u_0(z)dz+(2-i)(\tau_L ^i-\bar{\tau}_L ^i)\\
    <& \frac{1}{2}\int_{\bar{\tau}_L ^i}^{\tau_L ^i} \frac{{x}^2}{({t}-s)^2} ds +\frac{{\xi_L ^i}^2}{2 (\tau_L ^i+(i-2)u)}- \frac{\bar{{\xi}_L ^i}^2}{2 (\bar{\tau}_L ^i+(i-2)u)}\\& -\int_{(3-2i)\xi_L ^i}^{(3-2i)\bar{\xi}_L ^i} u_0(z)dz+(2-i)(\tau_L ^i-\bar{\tau}_L ^i )\leq 0, ~~\forall i.
    \end{aligned}
\end{equation*}
The last line follows from \eqref{inequality}.

 {\bf Step 2:} Suppose that at $(x_1, t),$  $U^i_L (x, t, \tau, \xi)$ attains minimum for $\tau=\bar{\tau}_L ^i$ and $\xi=\bar{\xi}_L ^i,~~\forall i $ where $\bar{\tau}_L ^i >\tau_L ^i.$ In this case the line joining $(x,t)$ and $(0,\tau_L ^i)$ will intersect the line joining $(x_1,t)$ and $(0,\bar{\tau}_L ^i)$ at a point $(\tilde{x}, \tilde{t})$ (say). Then by {\bf step-1} we get that 
$U^i _L (\tilde{x}, \tilde{t}, \tau, \xi)$ has the unique minimum for  $\tau=\bar{\tau}_L ^i$ and $\xi=\bar{\xi}_L ^i,$ and has the unique minimum for  $\tau=\tau_L ^i$ and $\xi=\xi_L ^i,~~\forall i .$ This implies $\tau_L ^i=\bar{\tau}_L ^i$ and $\xi_L ^i=\bar{\xi}_L ^i$. This is a contradiction to the fact that $\bar{\tau}_L ^i >\tau_L ^i.$  Therefore $\bar{\tau}_L ^i \leq \tau_L ^i.$

\noindent {\bf Step 3:} Suppose that at $(x_1, t),$  $U^i _L (x, t, \tau, \xi)$ attains minimum for $\tau=\bar{\tau}_L ^i$ and $\xi=\bar{\xi}_L ^i,~~\forall i.$ By {\textbf{ Step 2}}, $\bar{\tau}_L ^i \leq \tau_L ^i.$ We claim that $\bar{\xi}_L ^i \leq \xi_L ^i.$ Suppose not, then the line joining $(0, \tau_L ^i)$ and $(\xi_L ^i, 0)$ intersects the line joining $(0, \bar{\tau}_L ^i)$ and $(\bar{\xi}_L ^i, 0)$ at the point $(\tilde{x}, \tilde{t})$ (say).
Now we calculate the following
\begin{equation}\label{4.4}
    \begin{aligned}
    & U^i _L ({x},{t},\tau_L ^i,\xi_L ^i)-U^i _L ({x}, {t}, \tau_L ^i, \bar{\xi}_L ^i)\\
   = & \frac{x^2}{2(t-\tau_L ^i)}+\frac{{\xi_L ^i}^2}{2 (\tau_L ^i+(i-2)u)}+(2-i)(\tau_L ^i-u)+\int_{0}^{(3-2i)\xi_L ^i} u_0(z)dz\\&-\frac{x^2}{2(t-\tau_L ^i)}-\frac{\bar{{\xi}_L ^i}^2}{2 (\tau_L ^i+(i-2)u)}-(2-i)(\tau_L ^i-u)-\int_{0}^{(3-2i)\bar{\xi}_L ^i} u_0(z)dz \\
    =&\frac{{\xi_L ^i}^2}{2 (\tau_L ^i+(i-2)u)}-\frac{\bar{{\xi}_L ^i}^2}{2 (\tau_L ^i+(i-2)u)} -\int_{(3-2i)\xi_L ^i}^{(3-2i)\bar{\xi}_L ^i} u_0(z)dz\leq0 ,~~\forall i,
    \end{aligned}
\end{equation}since $U^i _L (x, t, \tau, \xi)$ attains minimum for $\tau=\tau_L ^i$ and $\xi=\xi_1.$ 
Now we assume that  $U^i _L (x_1, t, \tau, \xi)$ attains minimum for $\tau=\tau_2$ and $\xi=\xi_2> \xi_1.$ 
Then \begin{equation*}
    \begin{aligned}
    & U^i _L ({x_1},{t}, \bar{\tau}_L ^i,\bar{\xi}_L ^i)-U^i _L ({x_1}, {t}, \bar{\tau}_L ^i, \xi_L ^i)\\
   = & \frac{x_1^2}{2(t-\bar{\tau}_L ^i)}+\frac{\bar{{\xi}_L ^i}^2}{2 (\bar{\tau}_L ^i+(i-2)u)}+(2-i)(\bar{\tau}_L ^i-u)+\int_{0}^{(3-2i)\bar{\xi}_L ^i} u_0(z)dz\\&-\frac{x_1^2}{2(t-\bar{\tau}_L ^i)}-\frac{{\xi_L ^i}^2}{2 (\bar{\tau}_L ^i+(i-2)u)}-(2-i)(\bar{\tau}_L ^i-u)-\int_{0}^{(3-2i)\xi_L ^i} u_0(z)dz \\
    =& \frac{\bar{{\xi}_L ^i}^2}{2 (\bar{\tau}_L ^i+(i-2)u)}- \frac{{\xi_L ^i}^2}{2 (\bar{\tau}_L ^i+(i-2)u)} +\int_{(3-2i)\xi_L ^i}^{(3-2i)\bar{\xi}_L ^i} u_0(z)dz\leq0,~~\forall i. 
    \end{aligned}
\end{equation*}
Now since $\bar{\tau}_L ^i<\tau_L ^i,$ we have $$\frac{\bar{{\xi}_L ^i}^2-{\xi_L ^i}^2}{2 (\tau_L ^i+(i-2)u)}+\int_{(3-2i)\xi_L ^i}^{(3-2i)\bar{\xi}_L ^i} u_0(z)dz<\frac{\bar{{\xi}_L ^i}^2-{\xi_L ^i}^2}{2 (\bar{\tau}_L ^i+(i-2)u)} +\int_{(3-2i)\xi_L ^i}^{(3-2i)\bar{\xi}_L ^i} u_0(z)dz\leq0$$
which contradicts \eqref{4.4}. Hence, this contradiction leads to the fact that $\bar{\xi}_L ^i\leq \xi_L ^i.$
For the case $I=R$ we get the same steps as above and hence the result follows.

Now, since for $\xi\leq 0,U_L^3(x,t,\xi)=\frac{(x-\xi)^2}{2 t} + \int_{0}^{\xi} u_0(z)dz$ has minimum at $\xi=\xi_1$, we get the following \begin{equation*}
  \begin{split}
        &~~~~~U_L^3(x,t,\xi_1)-U_L^3(x,t,\xi_2)\\&=
        \frac{(x-\xi_1)^2}{2 t} + \int_{0}^{\xi_1} u_0(z)dz-\frac{(x-\xi_2)^2}{2 t} - \int_{0}^{\xi_2} u_0(z)dz\\&=\frac{x(\xi_2-\xi_1)}{t}+\frac{{\xi_1}^2-\xi_2^2}{2t}+\int_{\xi_2}^{\xi_1} u_0(z)dz\leq0
  \end{split}
\end{equation*} for any $\xi_2\le 0.$
Now we assume that $U_L^3(x_1,t,\xi)$ attains minimum for $\xi=\xi_2>\xi_1$, then we have \begin{equation*}
     \begin{split}
        &~~~~~U_L^3(x_1,t,\xi_2)-U_L^3(x_1,t,\xi_1)\\&=
        \frac{(x_1-\xi_2)^2}{2 t} + \int_{0}^{\xi_2} u_0(z)dz-\frac{(x_1-\xi_1)^2}{2 t} - \int_{0}^{\xi_1} u_0(z)dz\\&=\frac{x_1(\xi_1-\xi_2)}{t}+\frac{\xi_2^2-{\xi_1}^2}{2t}+\int_{\xi_1}^{\xi_2} u_0(z)dz\leq0.
  \end{split}
\end{equation*}
Adding the above two inequalities, we get \begin{equation*}
    \frac{(\xi_2-\xi_1)(x-x_1)}{t}\leq0
\end{equation*} which is a contradiction from our assumption. This contradiction leads to the fact that $\xi_2\leq \xi_1.$ Similarly, we can prove for the functional $U_R^3(x,t,\xi).$ This completes the proof of the lemma.
\end{proof}
\begin{theorem}\label{uniqueness}
For a.e. $(x,t)\in \mathbb{R}\times (0, \infty),$ the functionals $U_I^{i},$  where $i\in \{1,2,  3\}$ and $I\in \{L, R\}$  have unique mininimizers in their variables.
\end{theorem}
\begin{proof}
  Consider $U^1 _L (x, t, \tau, \xi).$  Let $\tau_{*} (x,t), \xi_{*}(x,t)$ denote the lowest point on the t-axis and the leftmost point on the x-axis, respectively, where $U^1 _L (x, t, \tau, \xi)$ attains minimum. Similarly, let $\tau^{*} (x,t), \xi^{*}(x,t)$ denote the highest point on the t-axis and the rightmost point on the x-axis, respectively, where $U^1 _L (x, t, \tau, \xi)$ attains a minimum. For fixed $t$, employing the lemma\eqref{nonintersecting property}, we find that $\xi^{*}(x,t)$, $ \xi_{*}(x,t)$ increase monotonically and  $\tau^{*} (x,t), \tau_{*} (x,t)$ decrease monotonically. So for a.e. $(x,t)\in \mathbb{R}\times (0, \infty),$ the functions $\xi^{*}(x,t)$, $ \xi_{*}(x,t)$, $\tau^{*} (x,t)$ and  $ \tau_{*} (x,t)$ are continuous. It is clear that at the point of continuity of  $\xi^{*}(x,t)$ and  $ \xi_{*}(x,t)$, we have $\xi^{*}(x,t)= \xi_{*}(x,t).$ Similarly, at the point of continuity of  $\tau^{*}(x,t)$ and  $ \tau_{*}(x,t)$, we have $\tau^{*}(x,t)= \tau_{*}(x,t).$  This proves that for a.e. $(x,t)\in \mathbb{R}\times (0, \infty),$ the minimizer $\tau$ and $\xi$ is unique for $U^1 _L (x, t, \tau, \xi).$  Similarly, we can handle the other functionals $U_I^{i},$  where $i\in \{1,2,  3\}$ and $I\in \{L, R\}$. This completes the proof.
\end{proof}
\begin{lemma}
    Given any $t>0,$ there exist $x_1(t)$ and $x_2(t)$ such that $0\leq x_2(t)\leq x_1(t)$ and \begin{equation*}
        U_R(x,t)=\begin{cases}
              U_R^2(x,t)+t,& 0\leq x\leq x_2(t)\\
              U_R^1(x,t)+t,& x_2(t)\leq x\leq x_1(t)\\
              U_R^3(x,t)+t,&x\geq x_1(t).
        \end{cases}
    \end{equation*}
\end{lemma}
\begin{proof}
     \textbf{Step-1:} $$\min \{U_R^1(0,t),U_R^2(0,t),U_R^3(0,t)\}=\min \{U_R^1(0,t),U_R^2(0,t)\}$$ since $$U_R^ 1(0, t)=\min_{\substack{0 \le \tau < t \\ 0 \le u < \tau \\ \xi \ge 0}}\Big[ U_R^1(0,t,\tau,\xi,u) \Big],U_R^3(0,t)= \min_{\xi \geq 0}\Big[U_R^3(0,t,\xi)\Big],$$ where $U_R^1(x,t,\tau,\xi,u),U_R^3(x,t,\xi)$ are defined  in \eqref{notation}, so at $u=0, \tau=t,$ we have $U_R^1(0,t,\tau,\xi,u)$ is same as the functional $U_R^3(0,t,\xi)$ which implies $$U_R^1(0,t)\leq U_R^3(0,t).$$ Similarly, we can prove $U_R^2(0,t)\leq U_R^3(0,t).$

    \noindent
    \textbf{Step-2:} Let $$\sup\{y_1:\min[U_R^1(x,t),U_R^2(x,t),U_R^3(x,t)]\}=\min[U_R^1(x,t),U_R^2(x,t)] ~~\forall x \in [0,y_1]\} =x_1(t)$$ for some $x_1(t)\geq 0.$ We claim that
    $$\forall x>x_1(t),~\min[U_R^1(x,t),U_R^2(x,t),U_R^3(x,t)]=U_R^3(x,t)$$. Let  $U_R^3(x,t,y)$ attains minimum at $y=y_1$. Also, let $(\bar{x},\bar{t})$ be any point on the line segment joining $(x,t)$ and $(y_1,0)$. We will show that at $(\bar{x},\bar{t}),$$$\min[U_R^1(x,t),U_R^2(x,t),U_R^3(x,t)]=U_R^3(x,t).$$ On the contrary, suppose that for some $\tau_1,\xi_1,u_1$ the following holds.
   \begin{equation*}
       \frac{\bar{x}^2}{2(\bar{t}-\tau_1)}+\frac{\xi_1 ^2}{2 (\tau_1-u_1)}-u_1 +\int_{0}^{\xi_1} u_0(z)dz\leq\frac{(\bar{x}-y_1)^2}{2\bar{t}} + \int_{0}^{y_1} u_0(z)dz.
   \end{equation*}
Now, \begin{equation*}
\begin{aligned}
     &\frac{x^2}{2(t-\tau_1)}+\frac{\xi_1 ^2}{2 (\tau_1-u_1)}-u_1 +\int_{0}^{\xi_1} u_0(z)dz-\frac{(x-y_1)^2}{2t} -\int_{0}^{y_1} u_0(z)dz\\
     =&\frac{x^2}{2(t-\tau_1)}-\frac{(x-y_1)^2}{2t} +\frac{\xi_1 ^2}{2 (\tau_1-u_1)}-u_1 +\int_{y_1}^{\xi_1} u_0(z)dz
\end{aligned}
\end{equation*}
So, it is enough to verify that 
\begin{equation}\label{inequality 1}
    \frac{x^2}{2(t-\tau_1)}-\frac{(x-y_1)^2}{2t}\leq \frac{\bar{x}^2}{2(\bar{t}-\tau_1)}-\frac{(\bar{x}-y_1)^2}{2\bar{t}}.
\end{equation}
The above holds if \begin{equation}\label{inequality 2}
    \frac{x^2}{t-\tau_1}-\frac{(x-y_1)^2}{t}\leq \frac{\bar{x}^2}{\bar{t}-\tau_1}-\frac{(x-y_1)^2}{t^2}\bar{t}.
\end{equation}
Since the slopes of the two lines joining $(x,t), (y_1,0)$ and $(\bar{x},\bar{t}),(y_1,0)$ are equal, we get $\frac{x-y_1}{t}=\frac{\bar{x}-y_1}{\bar{t}}.$ Employing this equality in \eqref{inequality 1} we get \eqref{inequality 2}.
Again, the above holds if $$\frac{x^2}{t-\tau_1}-  \frac{\bar{x}^2}{\bar{t}-\tau_1}\leq \frac{(x-y_1)^2}{t}\Big[1-\frac{\bar{t}}{t}\Big]$$
i.e. \begin{equation}\label{c1}
    \frac{x^2}{t-\tau_1} \leq \Big(\frac{x-y_1}{t}\Big)^2(t-\bar{t})+(\bar{t}-\tau_1)\Big(\frac{\bar{x}}{\bar{t}-\tau_1}\Big)^2
\end{equation}

 Now, we have \begin{equation}\label{c2}
  \begin{aligned}
         &\Big(\frac{x-y_1}{t}\Big)^2(t-\bar{t})+(\bar{t}-\tau_1)\Big(\frac{\bar{x}}{\bar{t}-\tau_1}\Big)^2\\=&(t-\tau_1)\Big[\Big(\frac{x-y_1}{t}\Big)^2\frac{t-\bar{t}}{t-\tau_1}+\frac{\bar{t}-\tau_1}{t-\tau_1}\Big(\frac{\bar{x}}{\bar{t}-\tau_1}\Big)^2\Big]\\ >&(t-\tau_1)\Big[\frac{x-\bar{x}}{t-\bar{t}}\times \frac{t-\bar{t}}{t-\tau_1}+\frac{\bar{t}-\tau_1}{t-\tau_1}\times\frac{\bar{x}}{\bar{t}-\tau_1}\Big]^2\\&=(t-\tau_1)\Big[\frac{x}{t-\tau_1}\Big]^2=\frac{x^2}{t-\tau_1}
  \end{aligned}
 \end{equation}
where the third line follows from the equality of slopes of the lines joining $(x,t),(y_1,0)$ and $(x,t),(\bar{x},\bar{t})$ and the property of convex function.
Therefore, clearly \eqref{c2} contradicts \eqref{c1}. Hence our claim follows.

Therefore there exist $x_1(t),x_2(t)>0$ such that $U_R(x,t)=U_R^2(x,t)+t,~~~\forall x\in(0,x_2(t)),$ $U_R(x,t)=U_R^1(x,t)+t,~~~\forall x\in(x_2(t),x_1(t)),$
$U_R(x,t)=U_R^3(x,t)+t,~~~\forall x>x_1(t),$
and for a fixed $t.$ Hence the lemma follows.
\end{proof}
Similarly for $U_L(x,t)$ we can prove the following lemma.
\begin{lemma}
      Given any $t>0,$ there exist $y_1(t)$ and $y_2(t)$ such that $y_1(t)\leq y_2(t)\leq 0$ and \begin{equation*}
        U_L(x,t)=\begin{cases}
              U_L^1(x,t),& 0\geq x\geq y_2(t)\\
              U_L^2(x,t),& y_2(t)\geq x\geq y_1(t)\\
              U_L^3(x,t),& x\leq y_1(t).
        \end{cases}
    \end{equation*}
\end{lemma}
\begin{lemma}
\label{derivative of U_R}
    The spatial derivative of $U_R(x,t)$ for a.e. $(x,t)\in (0,\infty)\times (0,\infty)$ is as follows.
    \begin{equation*}
        U_{R_x}(x,t)=\begin{cases}
               \frac{x}{t-\tau_2(x,t)},&0\leq x\leq x_2(t)\\
               \frac{x}{t-\tau_1(x,t)},&x_2(t)\leq x\leq x_1(t)\\
             \frac{x-\xi_1(x,t)}{t},&x\geq x_1(t).
        \end{cases}
    \end{equation*} where $U_R^i(x,t,\tau,\xi,u),U_R^3(x,t,\tau,\xi)$ achieve minimum at $\tau_i(x,t), \xi_1(x,t)$  for $i=1,2$ respectively.
    Similarly, the spatial derivative of $U_L(x,t)$ w.r.t. $x$ is given by
     \begin{equation*}
        U_{L_x}(x,t)=\begin{cases}
             \frac{x}{t-\bar{\tau}_1(x,t)} ,& 0\geq x\geq y_2(t)\\
               \frac{x}{t-\bar{\tau}_2(x,t)},& y_2(t)\geq x\geq y_1(t)\\
             \frac{x-\bar{\xi}_1(x,t)}{t} ,& x\leq y_1(t)
        \end{cases}
    \end{equation*}  where $U_L^i(x,t,\tau,\xi,u),U_L^3(x,t,\tau,\xi)$ achieve minimum at $\bar{\tau}_i(x,t), \bar{\xi}_1(x,t)$  for $i=1,2$ respectively.
\end{lemma}
\begin{proof}
     We find the derivative of $U_R^1(x,t)$ w.r.t. $x.$ 
Let $x\in (y_1,y_2)$ and $U_R^1(x,t)$ achieves minimum at a unique $\tau_1(x,t).$ Then for $x_1,x_2\in(y_1,y_2), x_1<x<x_2$ we have 
\begin{equation*}
    \begin{aligned}
       & U_R^1(x_2,t)-U_R^1(x_1,t)\\\leq &  \frac{x_2^2}{2(t-\tau_1(x_1,t))}+\frac{\xi_1(x_1,t) ^2}{2 (\tau_1(x_1,t)-u_1(x_1,t))}-u_1(x_1,t) +\int_{0}^{\xi_1(x_1,t)} u_0(z)dz\\&-\frac{x_1^2}{2(t-\tau_1(x_1,t))}-\frac{\xi_1(x_1,t) ^2}{2 (\tau_1(x_1,t)-u_1(x_1,t))}+u_1(x_1,t) -\int_{0}^{\xi_1(x_1,t)} u_0(z)dz\\=&\frac{x_2^2-x_1^2}{2(t-\tau_1(x_1,t))}
    \end{aligned}
\end{equation*}
which implies \begin{equation*}
\frac{ U_R^1(x_2,t)-U_R^1(x_1,t)}{x_2-x_1}\leq \frac{x_2+x_1}{2(t-\tau_1(x_1,t))}
\end{equation*}
Taking limit as $x_1,x_2 \to x,$ we get \begin{equation}\label{derivative 1}
    \displaystyle{\lim _{x_1,x_2 \to x}\frac{ U_R^1(x_2,t)-U_R^1(x_1,t)}{x_2-x_1}}\leq \frac{x}{t-\tau_1(x,t)}.
\end{equation}
The last line follows from the uniqueness of $\tau_1(x,t).$
Similarly we can prove that \begin{equation}\label{derivative 2}
    \displaystyle{\lim _{x_1,x_2 \to x}\frac{ U_R^1(x_2,t)-U_R^1(x_1,t)}{x_2-x_1}}\geq \frac{x}{t-\tau_1(x,t)}.
\end{equation}
Therefore from \eqref{derivative 1} and \eqref{derivative 2} we get \begin{equation*}
    \frac{\partial}{\partial x}\Big(U_R^1(x,t)\Big)=\frac{x}{t-\tau_1(x,t)}.
\end{equation*} 
Similarly, we can calculate $  \frac{\partial}{\partial x}\Big(U_R^2(x,t)\Big)=\frac{x}{t-\tau_2(x,t)}$ and $  \frac{\partial}{\partial x}\Big(U_R^3(x,t)\Big)=\frac{x-\xi_1(x,t)}{t}$ where $U_R^2(x,t),U_R^3(x,t)$ achieve minimum at $\tau_2(x,t), \xi_1(x,t)$ respectively.
Similarly, we can also prove for $ U_{L_x}(x,t)$.
\end{proof}
\begin{theorem}\label{solution of inviscid equation}
    The vanishing viscosity limit of the solution of the viscous equation \eqref{viscous equation}  is a weak (distributional) solution of the inviscid equation \eqref{inviscid equation}, where $u(x,t)$ is given by \begin{equation*}
        u(x,t)=\begin{cases}
            \frac{\partial}{\partial x}U_R,&x>0\\
             \frac{\partial}{\partial x}U_L,&x<0.
        \end{cases}
    \end{equation*}
  The explicit formula for $\frac{\partial}{\partial x}U_R$ and $\frac{\partial}{\partial x}U_L$ are given in Lemma\eqref{derivative of U_R}. 
\end{theorem}
\noindent

\begin{proof}
    Let $u^{\epsilon}(x,t)$ be the solution of \eqref{viscous equation}. We show that the vanishing viscosity limit of $u^{\epsilon}(x,t)$ is a solution of \eqref{inviscid equation}. First we consider the right-half domain $\{x>0\}$. Then we have \begin{equation}\label{limit}
        \displaystyle{\lim_{\epsilon \to 0}} u^{\epsilon}(x,t)= \displaystyle{\lim_{\epsilon \to 0}}\frac{\partial}{\partial x}\Big[-2\epsilon\log (R^{\epsilon}(x,t))\Big].
    \end{equation}
    Let $\phi\in C_c^{\infty} \Big((0,\infty)\times (0,\infty)\Big).$ Then 
    \begin{equation}\label{weak derivative 1}
        \int_0^{\infty}\Big[-2\epsilon\log (R^{\epsilon}(x,t))\Big]\phi_xdx=-\int_0^{\infty}\frac{\partial}{\partial x}\Big[-2\epsilon\log (R^{\epsilon}(x,t))\Big]\phi dx.
    \end{equation}
    Now by the Dominated convergence theorem, 
    \begin{equation*}
            \displaystyle{\lim_{\epsilon \to 0}}\int_0^{\infty}\Big[-2\epsilon\log (R^{\epsilon}(x,t))\Big]\phi_xdx=\int_0^{\infty} \displaystyle{\lim_{\epsilon \to 0}}\Big[-2\epsilon\log (R^{\epsilon}(x,t))\Big]\phi_xdx.
    \end{equation*} 
    Then from \eqref{weak derivative 1} and theorem \eqref{vanishing viscosity limit} we get
    \begin{equation*}
        -\displaystyle{\lim_{\epsilon \to 0}}\int_0^{\infty}\frac{\partial}{\partial x}\Big[-2\epsilon\log (R^{\epsilon}(x,t))\Big]\phi dx=\int_0^{\infty}U_R(x,t)\phi_xdx.
    \end{equation*}
    Again by Dominated convergence theorem, \begin{equation}\label{weak derivative 2}
        -\int_0^{\infty}\displaystyle{\lim_{\epsilon \to 0}}\frac{\partial}{\partial x}\Big[-2\epsilon\log (R^{\epsilon}(x,t))\Big]\phi dx=-\int_0^{\infty}\frac{\partial}{\partial x}\Big[U_R(x,t)\Big]\phi dx.
    \end{equation}
    Then combining \eqref{limit}, \eqref{weak derivative 2}, we get $$ \displaystyle{\lim_{\epsilon \to 0}} u^{\epsilon}(x,t)=\frac{\partial}{\partial x}\Big[U_R(x,t)\Big]=u(x,t).$$ This $u(x,t)$ satisfies the inviscid equation \eqref{inviscid equation} (see \cite{MR2028700}).

\end{proof}

\bibliographystyle{plain}	
\nocite{*}
\bibliography{Reference}
\end{document}